\documentclass{amsart}

\usepackage{slashed}
\usepackage{mathtools}
\usepackage{float}
\usepackage{amsmath}
\usepackage{amsfonts}
\usepackage{amssymb}
\usepackage{graphicx}
\usepackage{hyperref}
\usepackage{mathrsfs}
\usepackage{pb-diagram}
\usepackage{epstopdf}
\usepackage{amsmath,amsfonts,amsthm,enumerate,amscd,latexsym,curves}
\usepackage{bbm}
\usepackage{epsfig,epsf,xypic,epic}
\usepackage{caption}
\usepackage{subcaption}
\usepackage{tikz}
\usepackage{geometry}
\usetikzlibrary{matrix,arrows,decorations.pathmorphing}
\usetikzlibrary{cd}
\usepackage{tikz-cd}
\usepackage{enumitem}
\makeatletter
\newcommand{\xslalph}[1]{\expandafter\@xslalph\csname c@#1\endcsname}
\newcommand{\@xslalph}[1]{
    \ifcase#1\or a\or b\or c\or \or d\or e\or f\or g\or h\or i
    \or j\or k\or l\or m\or n\or o\or p\or r\or s\or 
    \or t\or u\or v\or z\or \v{z}
    \else\@ctrerr\fi
}
\AddEnumerateCounter{\xslalph}{\@xslalph}{m}
\makeatother

\usetikzlibrary{matrix,arrows,decorations.pathmorphing}

\theoremstyle{theorem}
\newtheorem{theorem}{Theorem}[section]

\newtheorem{lemma}[theorem]{Lemma}
\newtheorem{corollary}[theorem]{Corollary}
\newtheorem{proposition}[theorem]{Proposition}

\theoremstyle{definition}
\newtheorem{definition}[theorem]{Definition}

\newtheorem{remark}[theorem]{Remark}

\numberwithin{equation}{section}

\newcommand{\Id}{\mathrm{Id}}

\newcommand{\End}{\mathrm{End}}

\newcommand{\Li}{\mathrm{Li}}

\newcommand{\std}{\mathrm{std}}

\newcommand{\tw}{\mathrm{tw}}

\newcommand{\Hom}{\mathrm{Hom}}
\newcommand{\dbar}{\ol{\partial}}

\newcommand{\bb}[1]{\mathbb{#1}}

\newcommand{\cu}[1]{\mathcal{#1}}

\newcommand{\ol}[1]{\overline{#1}}

\newcommand{\ul}[1]{\underline{#1}}

\def\Hom{\text{\rm Hom}}

\begin{document}

\title[Brane quantization of $A_n$-resolutions]{Brane quantization of $A_n$-resolutions}

\author[Suen and Yau]{Yat-Hin Suen and YuTung Yau}

\address{Department of Mathematics, National Cheng Kung University}
\email{yhsuen@gs.ncku.edu.tw}
\address{Kavli Institute for the Physics and Mathematics of the Universe (WPI), The University of Tokyo Institutes for Advanced Study, The University of Tokyo, Kashiwa, Chiba 277-8583, Japan}
\email{yu-tung.yau@ipmu.jp}

\thanks{}

\maketitle

\begin{abstract}
    We extend the study of brane quantization via SYZ mirror symmetry to the setting of singular fibers, building on recent joint work with Chan, Leung, and Li in the semi-flat case. We consider a crepant resolution $X\to\bb{C}^2/\bb{Z}_{n+1}$ of the $A_n$-singularity, whose mirror $\check{X}$ is also realized as a resolution of $\bb{C}^2/\bb{Z}_{n+1}$. For each level $k\in\bb{Z}_{>0}$, we construct a space filling coisotropic A-brane $\cu{B}_{cc}^{(k)}$ of $(X,k\omega)$ and determine its mirror B-brane $\check{\cu{B}}_{cc}^{(k)}$ via fiberwise geometric quantization. We then define the endomorphism algebra $\Hom_A(\cu{B}_{cc}^{(k)},\cu{B}_{cc}^{(k)})$ by gluing analytic quantum tori using wall-crossing formulas and establish a mirror isomorphism $\Hom_A(\cu{B}_{cc}^{(k)},\cu{B}_{cc}^{(k)})\cong\Hom_B(\check{\cu{B}}_{cc}^{(k)},\check{\cu{B}}_{cc}^{(k)})$.
\end{abstract}

\section{Introduction}
\label{Section: Introduction}
Cosiotropic A-branes were introduced by Kapustin--Orlov \cite{KapOrl2003} to enlarge the Fukaya category so as to accommodate Kontsevich's homological mirror symmetry conjecture \cite{HMS}. Among them, space-filling coisotropic A-branes are distinguished by the holomorphic symplectic structure $\Omega_X$ on the ambient manifold $X$ induced by their curvature, and play a central role of the brane quantization proposal of Gukov--Witten \cite{GukWit2009}. In their framework, one considers a distinguished space-filling A-brane $\cu{B}_{cc}$ \footnote{They refer to $\cu{B}_{cc}$ as the \emph{canonical coisotropic A-brane}, which explains the notation.} on a complexification $(X, \Omega_X)$ of a K\"ahler manifold $(M, \omega_M)$. They propose that the endomorphism algebra $\Hom_A(\cu{B}_{cc},\cu{B}_{cc})$ provides a holomorphic deformation quantization of $(X, \Omega_X)$, 
while for a Lagrangian A-brane $\cu{B}$ supported on $M$, the morphism space $\Hom_A(\cu{B}, \cu{B}_{cc})$ gives the geometric quantization of $(M, \omega_M)$. These two structures are linked by a folklore action ``$DQ\curvearrowright GQ$'', implemented by the composition map
$$\operatorname{Hom}_A(\cu{B}_{cc},\cu{B}_{cc})\otimes \operatorname{Hom}_A(\cu{B},\mathcal{B}_{cc})\to \operatorname{Hom}_A(\cu{B},\cu{B}_{cc}).$$

In recent joint work \cite{ChaLeuLiSueYau2026} with Chan, Leung, and Li, we make this picture precise in the semi-flat SYZ setting and show its compatibility with SYZ mirror symmetry \cite{ChanLeung10, Leung05, LeuYauZas2000, SYZ}. We consider a $\bb{Z}_{>0}$-family of semi-flat SYZ mirror pairs
\begin{center}
	\begin{tikzcd}
		X_0^{(k)}=(T^*B^{sm}/T_\mathbb{Z}^*B^{sm},k\omega_0) \ar[rd, "p_0"'] & & \check{X}_0^{(k)}=(TB^{sm}/ T_\mathbb{Z}B^{sm},\check{J}_{0, k}) \ar[ld, "\check{p}_0^{(k)}"]\\
		& B^{sm}
	\end{tikzcd}
\end{center}
together with a mirror pair of semi-affine space-filling branes $( \cu{B}_{cc, 0}^{(k)}, \check{\cu{B}}_{cc, 0}^{(k)} )$ under the SYZ transform of \cite{ChaLeuZha2018}, and require that $\{\cu{B}_{cc, 0}^{(k)}\}_{k \in \mathbb{Z}_{>0}}$ determines a fixed holomorphic symplectic form on $X_0 = T^*B^{sm}/T_\mathbb{Z}^*B^{sm}$, up to the expected rescaling by $k$. We realize $\Hom_A(\cu{B}_{cc, 0}^{(k)},\cu{B}_{cc, 0}^{(k)})$ as a non-formal holomorphic deformation quantization of $X_0$, constructed by gluing analytic quantum tori.
By incorporating the ``$DQ \curvearrowright GQ$'' action on Lagrangian fibres of $X_0$ via family Toeplitz operators, we establish for each $k$ a graded algebra isomorphism
\begin{equation}
    \label{eqn: semi-flat mirror transform}
    \Phi_0^{(k)}:\Hom_A(\cu{B}_{cc, 0}^{(k)},\cu{B}_{cc, 0}^{(k)})\to\Hom_B(\check{\cu{B}}_{cc, 0}^{(k)},\check{\cu{B}}_{cc, 0}^{(k)}),
\end{equation}
where $\Hom_B(\check{\cu{B}}_{cc, 0}^{(k)},\check{\cu{B}}_{cc, 0}^{(k)})$ denotes the endomorphism algebra in the derived category of $\check{X}_0^{(k)}$.

Despite this success, the scope of \cite{ChaLeuLiSueYau2026} is restricted to the semi-flat setting. The main obstacle is that, beyond this setting, a mirror transform for coisotropic A-branes at the level of objects is not presently available.

In this article, we take a first step toward extending the brane quantization program to non-semi-flat SYZ fibrations, namely those with singular fibres. As a model case, we consider crepant resolutions of $A_n$-singularity. Let
$$X:=\{(x,y,z)\in\bb{C}^2\times\bb{C}^{\times}:xy=f(z)\},$$
where $f$ is a degree $n+1$ polynomial with $n+1$ distinct roots, and equip $X$ with the symplectic form $\omega$ restricted from the standard one on $\bb{C}^2\times\bb{C}^{\times}$. Then $(X,\omega)$ admits a Lagrangian torus fibration $p:X\to B$ onto $B=\bb{R}^2$ with $n+1$ nodal fibers. Its mirror manifold $\check X$ (and more generally the mirror $\check X^{(k)}$ of $(X,k\omega)$) is obtained from the semi-flat mirror by incorporating wall-crossing corrections.

Over the smooth locus $B^{sm} \subset B$, the pair $(X,\check X)$ reduces to the semi-flat situation studied in \cite{ChaLeuLiSueYau2026}. This allows us to begin with a distinguished mirror pair of semi-affine space-filling branes $(\cu{B}_{cc, 0}^{(k)}, \check{\cu{B}}_{cc, 0}^{(k)})$ on the semi-flat parts. A key observation is that these branes admit natural extensions across the singular fibers. Carrying out this extension requires incorporating the wall-crossing data, which in turn induces corrections to both the deformation quantization algebra $\Hom_A (\cu{B}_{cc, 0}^{(k)}, \cu{B}_{cc, 0}^{(k)})$ and the semi-flat mirror transform $\Phi_0^{(k)}$.

Our main result is that these corrections give rise to space-filling branes $\cu{B}_{cc}^{(k)}$ on $(X, k\omega)$ and $\check{\cu{B}}_{cc}^{(k)}$ on $\check{X}^{(k)}$, whose endomorphism algebras continue to be related by mirror symmetry:
\begin{theorem}[=Theorem \ref{thm:main theorem}]
    For every $k \in \mathbb{Z}_{>0}$, the instanton-corrected mirror transform is a graded algebra isomorphism
    \begin{equation}
        \label{eqn: mirror transform}
        \Phi^{(k)}: \Hom_A(\cu{B}_{cc}^{(k)},\cu{B}_{cc}^{(k)}) \overset{\cong}{\longrightarrow} \Hom_B(\check{\cu{B}}_{cc}^{(k)},\check{\cu{B}}_{cc}^{(k)}).
    \end{equation}
\end{theorem}

As in \cite{ChaLeuLiSueYau2026}, by identifying a point $(x,\check{y})\in\check{X}_0$ with a Lagrangian brane $\cu{B}_{(x,\check{y})}=(F_x,\cu{L}_{\check{y}})$ that supports on a Lagrangian fiber $F_x\subset X$, the fiber $\check{E}_{(x,\check{y})}^{(k)}$ is given by the the K\"ahler polarized Hilbert space
$$\Hom_A(\cu{B}_{(x,\check{y})},\cu{B}_{cc}^{(k)}):=H^0(F_x,L^{\otimes k}|_{F_x}\otimes\cu{L}_{\check{y}}^*\otimes\sqrt{K_x}),$$
where $\sqrt{K_x}$ is a half-form bundle on $F_x$ with respect to an axillary choice of flat complex structure. As a corollary of our main theorem, we obtain the brane quantization of $A_n$-resolutions.

\begin{corollary}
    For each level $k\in\bb{Z}_{>0}$ and each $(x, \check{y}) \in \check{X}_0$, there is an action $$\Hom_A(\cu{B}_{cc}^{(k)},\cu{B}_{cc}^{(k)})\otimes\Hom_A(\cu{B}_{(x,\check{y})},\cu{B}_{cc}^{(k)})\to\Hom_A(\cu{B}_{(x,\check{y})},\cu{B}_{cc}^{(k)}).$$
\end{corollary}

The resulting construction furnishes a mirror transform for non-Lagrangian A-branes at both the object and morphism levels in the $A_n$-resolution setting. As a prototypical non-semi-flat example, it demonstrates how the semi-flat brane quantization picture can be extended across singular fibres and suggests a guiding principle for further extensions of the brane quantization program.

\subsection{Key ideas and organization of the article}

We conclude this introduction with a brief overview of the main ideas underlying the wall-crossing correction and the organization of the article.

Section \ref{Section 2} reviews the geometric background and establishes the notation used throughout the article.

Section \ref{sec: A brane} is devoted to the A-side correction. We observe that the monodromy of the affine complex coordinates induced by the semi-flat A-brane $\cu{B}_{cc,0}$ agrees with the monodromy of the complex coordinates on the semi-flat mirror manifold $\check{X}_0$. By correcting these coordinates via the same wall-crossing formula as in the construction of $\check{X}$, we show that $\cu{B}_{cc,0}$ naturally extends to a space-filling A-brane $\cu{B}_{cc}$
on $(X,\omega)$.

We then correct $\Hom_A (\cu{B}_{cc, 0}^{(k)}, \cu{B}_{cc, 0}^{(k)})$ by modifying the gluing of analytic quantum tori using wall-crossing factors with a non-standard feature: they are inherently multi-valued, since they involve $2k$-th roots. After choosing local branches, different choices yield canonically isomorphic algebras, so that the local constructions glue consistently to a sheaf 
$\cu{A}_X^{(k)}$ of algebras on $B$. We propose that
$$\Hom_A(\cu{B}_{cc}^{(k)},\cu{B}_{cc}^{(k)}):=H^{\bullet}(B,\cu{A}_X^{(k)}).$$
To align with Gukov--Witten proposal, we further show that for each level $k\in\bb{Z}_{>0}$, there is a deformation quantization of $\cu{O}_X$ relative to $B$ (Definition \ref{def:relative DQ}) $\cu{A}_{X,\hbar}^{(k)}$ whose specialization at $\hbar=-\frac{\pi}{k}$ is $\cu{A}_X^{(k)}$.

Section \ref{sec:mirror B brane} addresses the B-side correction. The semi-flat mirror B-brane $\check{\cu{B}}_0^{(k)}$ carries an underlying unitary vector bundle $(\check{E}_0^{(k)},\nabla^{\check{E}_0^{(k)}})$ under the construction of \cite{ChaLeuZha2018}. A simple calculation shows the connection $\nabla^{\check{E}_0^{(k)}}$, or equivalently, the unitary structure, does not extend to $\check{X}^{(k)}$. Nevertheless, the holomorphic structure of $\check{E}_0^{(k)}$ does! This gives the mirror B-brane $\check{\cu{B}}_{cc}^{(k)}:=(\check{X}^{(k)},\check{E}^{(k)})$ of $\cu{B}_{cc}^{(k)}$, and we define
$$\Hom_B(\check{\cu{B}}_{cc}^{(k)},\check{\cu{B}}_{cc}^{(k)}):= H^{\bullet}(\check{X}^{(k)},\End(\check{E}^{(k)})).$$

Section \ref{sec:mirror theorem} contains the proof of our main theorem. In the semi-flat case \cite{ChaLeuLiSueYau2026}, despite its generality, the corresponding mirror theorem is established using the Weil–Berezin transform, which is rather abstract and not explicit. In contrast, in our explicit model we construct the isomorphism $\Phi^{(k)}$ in \eqref{eqn: mirror transform} by classifying local sections and generators of $\End(\check{E}^{(k)})$. Locally, $\Phi^{(k)}$ is defined by sending generators of the quantum tori to corresponding local sections of $\End(\check{E}^{(k)})$, yielding an algebra isomorphism. Finally, to show that $\Phi^{(k)}$ is globally well-defined, we verify its compatibility with the wall-crossing formulas on both sides. It suffices to check this compatibility on the generators of the quantum tori.

In Appendix \ref{app:Toeplitz}, we prove that our mirror map agrees locally with the one constructed in \cite{ChaLeuLiSueYau2026}. This appendix is logically independent of the main body of the paper and may be skipped by readers who are interested only in the mirror symmetry statement.

\subsection*{Acknowledgement} 
Both authors would like to thank Kwokwai Chan and Nai Chung Conan Leung for useful discussions and their interest in this work. The work of Y.-H. Suen was supported by the National Science and Technology Council (Project No. 113-2115-M-006-016-MY2) and partially supported by the Yushan Fellow Program by the Ministry of Education (MOE), Taiwan. (MOE-114-YSFMS-0005-001-P1). Y. Yau was supported by World Premier International Research Center Initiative (WPI), MEXT, Japan.

\section{Mirror symmetry of $A_n$-resolutions}
\label{Section 2}

This section reviews some basics on mirror symmetry of $A_n$-resolutions. We basically follow \cite{Cha2013} but with a slight change in notations and sign conventions in order to simplify a lot of sign issues in the later sections.

Let
$$X:=\{(x,y,z)\in\bb{C}^2\times\bb{C}^{\times}:xy=f(z)\},$$
where $f$ is a degree $n+1$ polynomial with $n+1$ distinct roots $a_0,\dots,a_n$ such that $|a_0| > \cdots > |a_n|$. Equip $X$ with the symplectic form $\omega$, obtained by restricting the standard symplectic form
$$\omega_{\std}:=\frac{\sqrt{-1}}{2}\left(dx\wedge d\ol{x}+dy\wedge d\ol{y}+\frac{dz\wedge d\ol{z}}{|z|^2}\right)$$
on $\bb{C}^2\times\bb{C}$ to $X$. It admits a proper Lagrangian fibration $p:X\to B:=\bb{R}^2$ given by
$$p(x,y,z):=\left(-\log|z|,\frac{1}{2}(|x|^2-|y|^2)\right).$$


We denote a general point in $B = \mathbb{R}^2$ by $(s, \lambda)$. This fibration has nodal singular fibers over the points $(s, \lambda) = (s_i,0)$, where $s_i:= -\log|a_i|$ for each $i=0,\dots,n$. In our setting, we have $s_0<\cdots<s_n$, and the base space $B$ is covered by the open subsets
$$B_i:=(s_{i-1},s_{i+1})\times\bb{R}, \quad i=0,\dots,n,$$
where $s_{-1}:=-\infty$ and $s_{n+1}:=+\infty$. The smooth locus $B^{sm} := B \backslash \{ (s_0, 0), ..., (s_n, 0) \}$ admits an integral affine structure induced by the above Lagrangian fibration, receiving monodromy around each point $(s_i, 0)$. To describe the monodromy, we further cover each $B_i \backslash \{ (s_i, 0) \}$ by two open subsets
\begin{align*}
U_i^-:=&\,B_i\backslash ((s_{i-1},s_i]\times\{0\}),\\
U_i^+:=&\,B_i\backslash ([s_i,s_{i+1})\times\{0\}).
\end{align*}
With respect to a set of affine coordinates $(x_{U_i^{\pm},1},x_{U_i^{\pm},2})$ on $U_i^{\pm}$, the counter-clockwise monodromy around each affine singularity is given by
\begin{equation}
    \label{Equation: monodromy matrix}
        \begin{pmatrix}
            1 & -1\\
            0 & 1
        \end{pmatrix}.
    \end{equation}
For notational convenience, we will often suppress the subscript $U_i^{\pm}$ indicating the choice of local chart when no confusion can arise. In particular, coordinates, local frames, and other local data will be denoted without this subscript except when discussing transition data on overlaps of open subsets.

Let $(y^1,y^2)$ denote the fiber coordinates of $X$ induced by the base coordinates $(x_1,x_2)$. In these local action-angle coordinates, the symplectic form on $X$ is locally given by
$$\omega =dx_1 \wedge dy^1+dx_2\wedge dy^2.$$

\subsection{The mirror manifold}

The mirror manifold of $(X,k\omega)$ is constructed as follows. Let
$$\check{p}^{(k)}:\check{X}^{(k)}_0:=TB^{sm}/ T_{\mathbb{Z}} B^{sm} \to B^{sm}$$
be the dual torus fibration. For each open subset $U \subset B^{sm}$, denote $(\check{p}^{(k)})^{-1}(U)$ by $\check{X}^{(k)}_U$. The total space $\check{X}^{(k)}_0$ is canonically diffeomorphic to $\check{X}_0 := TB^{sm}/ T_\mathbb{Z} B^{sm}$, and admits a canonical complex structure determined by the local complex coordinates on each $\check{X}^{(k)}_{U_i^{\pm}}$:
\begin{align*}
    (\check{z}_1, \check{z}_2) :=\, \left( x_1+\frac{\sqrt{-1}}{k}\check{y}_1, \,x_2+\frac{\sqrt{-1}}{k}\check{y}_2 \right),
\end{align*}
under the canonical identification $\check{X}^{(k)}_0 \cong \check{X}_0$ of the underlying smooth manifolds. Here $(\check{y}_1,\check{y}_2)$ denotes the fiber coordinates of $\check{X}_0$ induced by the base affine coordinates $(x_1,x_2)$. The resulting complex manifold $\check{X}^{(k)}_0$ is known as the \emph{semi-flat mirror manifold} of $(X, k\omega)$. On $U=U_i^{\pm}$, define
\begin{align*}
    (\check{u}, \check{w}) :=&\, (e^{2\pi k\check{z}^1}, e^{2\pi k\check{z}^2}).
\end{align*}
According to \eqref{Equation: monodromy matrix}, these complex coordinates have monodromy
\begin{equation}
    \label{eqn: monodromy on semi-flat mirror}
    \begin{dcases}
    \check{u}_{U_i^-}\mapsto\check{u}_{U_i^-}\check{w}_{U_i^-}^{-1}\\
    \check{u}_{U_i^+}\mapsto \check{u}_{U_i^+}\check{w}_{U_i^-}^{-1}\\
    \check{w}_{U_i^-}\mapsto \check{w}_{U_i^-}.
\end{dcases}
\end{equation}
Let $B^+ = \{ (s, \lambda) \in B: \lambda > 0 \}$ and $B^- = \{ (s, \lambda) \in B: \lambda < 0 \}$. Define
\begin{align*}
    B_i^\pm :=&\, B_i \cap B^\pm.
\end{align*}
Then $U_i^+\cap U_i^-=B_i^+\sqcup B_i^-$. The counter-clockwise monodromy \eqref{Equation: monodromy matrix} gives us the following semi-flat gluing
\begin{equation}\label{eqn:naive gluing}
    \begin{dcases}
    \check{u}_{U_i^+}=\check{u}_{U_i^-} & \text{ on } \check{X}_{B_i^+}^{(k)},\\
    \check{u}_{U_i^+}=\check{u}_{U_i^-}\check{w}_{U_i^-}^{-1} & \text{ on }\check{X}_{B_i^-}^{(k)},\\
    \check{w}_{U_i^+}=\check{w}_{U_i^-}=\check{w}_{U_{i+1}^-}=\check{w}_{U_{i+1}^+} & \text{ for all } i,\\
    \check{u}_{U_i^+}=\check{u}_{U_{i+1}^-} & \text{ for all } i.
\end{dcases}
\end{equation}

By studying the Floer theory of fibers of $p:X\to B$ \cite{Aur2007, Aur2009}, 
one finds that for each affine singularity $(s_i,0)$, there are two walls emitting out from it. They are given by $\{s_i\}\times\bb{R}_{>0},\{s_i\}\times\bb{R}_{<0}$, and in particular, all walls are parallel. The fiber over each point on the wall bounds Maslov index-zero holomorphic disks emitting out from the singular fiber. Such Floer-theoretic data are used to correct the gluing \eqref{eqn:naive gluing}, thereby producing a monodromy-free complex structure. The \emph{instanton-corrected mirror manifold} of $(X,k\omega)$ is obtained by modifying \eqref{eqn:naive gluing} with the corresponding \emph{wall-crossing factors}:
\begin{equation}\label{eqn: B gluing}
    \begin{dcases}
    \check{u}_{U_i^+}=\check{u}_{U_i^-} (1+\check{w}_{U_i^-}^{-1}),  & \text{ on } \check{X}_{B_i^+}^{(k)},\\
    \check{u}_{U_i^+}=\check{u}_{U_i^-}\check{w}_{U_i^-}^{-1}(1+\check{w}_{U_i^-}) & \text{ on } \check{X}_{B_i^-}^{(k)},\\
    \check{w}_{U_i^+}=\check{w}_{U_i^-}=\check{w}_{U_{i+1}^-}=\check{w}_{U_{i+1}^+} & \text{ for all }i,\\
    \check{u}_{U_i^+}=\check{u}_{U_{i+1}^-} & \text{ for all }i.
\end{dcases}
\end{equation}
Denote the resulting complex manifold by $\check{X}^{(k)}$ and the corresponding structure sheaf by $\cu{O}_{\check{X}^{(k)}}$. Note that $x_2 := x_{U_i^{\pm}, 2}$ and $\check{w}:=\check{w}_{U_i^{\pm}}$ are global functions of $\check{X}^{(k)}$, and we require that $x_2 = 0$ along the line $\lambda = 0$ in $B = \mathbb{R}^2$. This condition ensures that the coordinate function $\check{w}$ maps $\check{X}^{(k)}_{B^+}$ and $\check{X}^{(k)}_{B^-}$ into the two distinct connected components of $\{ a \in \mathbb{C}: | a | \neq 1\}$, 
thereby guaranteeing that the gluing relation \eqref{eqn: B gluing} is well defined (a geometric construction of $\check{w}$ is provided in \cite[page 10]{Cha2013}).

\begin{remark}
    The complex manifold $(\check{X}^{(k)},\cu{O}_{\check{X}^{(k)}})$ is isomorphic to the toric variety $X_{\Sigma}$ with fan $\Sigma$ generated by the rays $(0,1),(1,1),\dots,(n+1,1)$, which is a resolution of the $A_n$-singularity $\bb{C}^2/\bb{Z}_{n+1}$. See \cite{Cha2013}.
\end{remark}

\subsection{Coisotropic A-branes and SYZ transform}

Recall the following definition \cite{KapOrl2003, Gua2011}.

\begin{definition}
    A \emph{(rank 1) coisotropic A-brane} of a symplectic manifold $(M,\omega_M)$ is a coisotropic submanifold $C\subset M$ together with an $U(1)$-line bundle $(L,\nabla^L)$ such that
	\begin{itemize}
		\item [(a)] The curvature $-2\pi\sqrt{-1} F_{\nabla^L}$ vanishes on $\ker(\omega_M|_C)$. In particular, the $2$-form $F_{\nabla^L}$ descends to a map $F_{\nabla^L}:TC/\ker(\omega_M|_C)\to T^*C$.
		\item [(b)] The composition $I:=\omega_M^{-1}F_{\nabla^L}:TC/\ker(\omega_M|_C) \to TC/\ker(\omega_M|_C)$ defines a complex structure, i.e. $I^2=-\operatorname{Id}$.
	\end{itemize}
\end{definition}

Throughout this article, every coisotropic A-brane is assumed to be of \emph{rank} $1$ and supported on a \emph{closed} submanifold of the ambient symplectic manifold. We are mainly interested in coisotropic A-branes $(C, L, \nabla^L)$ of $(X, \omega)$ supported on non-Lagrangian submanifolds. By dimension reason, such an A-brane must be \emph{space-filling}, i.e. $C = X$. In this case, $X$ inherits a complex structure $I=\omega^{-1}F_{\nabla^L}$ from the brane structure, turning it into a holomorphic symplectic manifold with holomorphic symplectic form
$$\Omega_X:=F_{\nabla^{L}}+\sqrt{-1}\omega.$$
One can also scale the symplectic form $\omega$ by a positive integer $k\in\bb{Z}_{>0}$ and see that $(X,L^{\otimes k},\nabla^{L^{\otimes k}})$ is a coisotropic brane of $(X,k\omega)$.

In semi-flat SYZ framework (cf. \cite{ChaLeuZha2018, ChaLeuLiSueYau2026}), a B-brane is taken to be a Hermitian holomorphic vector bundle over a complex submanifold. For the purposes of this paper, however, we omit the Hermitian structure.

\begin{definition}
    \label{Definition: B-brane}
	A \emph{B-brane} of a complex manifold $\check{M}$ is a complex submanifold $\check{C}\subset\check{M}$ together with a holomorphic vector bundle $\check{E}$ on $\check{C}$.
\end{definition}

In \cite{ChaLeuZha2018}, Chan--Leung--Zhang introduced the notion of \emph{semi-affineness} of coisotropic $A$-branes and B-branes, and established a correspondence
$$\{\text{Semi-affine A-branes}\}\longleftrightarrow\{\text{Semi-affine B-branes}\}$$
via the family Nahm transform. Their construction applies more generally to higher-rank coisotropic A-branes and requires the corresponding B-branes to be equipped with Hermitian structures. In particular, a space-filling A-brane $\cu{B}_0 = (X_0, L_0, \nabla^{L_0})$ of $(X_0, \omega \vert_{X_0})$ is semi-affine if and only if its curvature $-2\pi\sqrt{-1} F_{\nabla^{L_0}}$ is a constant $2$-form along each fiber $F_x$ of $p \vert_{X_0}$, that is, in affine coordinates $(y^i)$ of $F_x$, there exist integers $h_{ij}\in\mathbb{Z}$ such that $F_{\nabla^{L_0}}|_{F_x}=\sum_{i,j}h_{ij}dy^i\wedge dy^j$. 

Applying the family Nahm transform to a semi-affine space-filling A-brane $\cu{B}_0$ produces not only a B-brane $\check{\cu{B}}_0 = (\check{C}_0, \check{E}_0)$ in the sense of Definition \ref{Definition: B-brane}, but also a unitary connection $\nabla^{\check{E}_0}$ on $\check{E}_0$ such that the triple $(\check{C}_0, \check{E}_0, \nabla^{\check{E}_0})$ satisfies the semi-affine condition of \cite{ChaLeuZha2018}. We call $\check{\cu{B}}_0$ the \emph{mirror B-brane} of $\cu{B}_0$.

Henceforth, we restrict our attention to the case when the mirror B-brane $\check{\cu{B}}_0$ is also space-filling, i.e. $\check{C}_0 = \check{X}_0$. This is equivalent to requiring that the restriction of $F_{\nabla^{L_0}}$ to every fiber of $p \vert_{X_0}: X_0 \to B^{sm}$ is non-degenerate \cite{ChaLeuZha2018}. This setting was studied in the joint work \cite{ChaLeuLiSueYau2026} with Chan, Leung and Li. In particular, it is shown that, after a suitable choice of local affine coordinates, the connection $\nabla^{L_0}$ can be written as
$$\nabla^{L_0}=d-2\pi\sqrt{-1} (x_1h^{-1}dx_2+(y^2+g^2(x))hd(y^1+g^1(x))),$$
for some $h\in\bb{Z}$ and closed 1-form $\alpha = g^1(x)dx_1 + g^2(x)dx_2$. Moreover, the integer $h$ is an invariant of the A-brane $\cu{B}_0$. As demonstrated in \cite[Example~3.3.3]{ChaLeuLiSueYau2026}, there exists a semi-flat space-filling A-brane on $X_0$ whose connection is given by the above expression with $h = 2$ and $\alpha = 0$. This local model will serve as the prototype for our construction away from the singular fibers. In Section \ref{sec: A brane}, we will modify the gluing of these local models in order to construct a space-filling A-brane on $X$.

For each positive integer $k\in\bb{Z}_{>0}$, accordingly, the same kind of constructions yields the mirror B-brane $\check{\cu{B}}_0^{(k)} = (\check{X}_0^{(k)}, \check{E}_0^{(k)})$ of $\cu{B}_0^{(k)} := (X_0, L_0^{\otimes k}, \nabla^{L_0^{\otimes k}})$, which is equpiped with a Hemritian structure. When $\alpha = 0$, the Chern connection $\nabla^{\check{E}_0^{(k)}}$ on the mirror bundle locally takes the form
$$\nabla^{\check{E}_0^{(k)}}=d-2\pi\sqrt{-1} \left(kx_1h^{-1}dx_2+\frac{1}{k}\check{y}_2h^{-1}d\check{y}_1\right)\Id_{\check{E}_0^{(k)}}$$
This formula will play a key role in Section \ref{sec:mirror B brane}, where we correct the mirror B-brane correspondingly.

\section{A space-filling A-brane on $X$ and its endomorphism algebra}
\label{sec: A brane}

The goal of this section is to construct a space-filling coisotropic A-brane of $(X, \omega)$ starting from a semi-affine one, and to present a mathematical model for its endomorphism algebra.

Subsection \ref{subsec:construction of Bcc} introduces a semi-affine coisotropic A-brane $\cu{B}_{cc, 0}$ of $(X_0, \omega |_{X_0})$, which already serves as a prototypical example within this class of objects. A key observation is that the monodromy of the complex structure $I_0$ on $X_0$, induced by $\cu{B}_{cc, 0}$, agrees with that of the semi-flat mirror of $X$. This suggests that $I_0$ can be corrected to a global complex structure on $X$. Subsection \ref{subsec:Construction of the space filling A-brane} promotes this correction to the level of the A-brane, thereby producing a space-filling coisotropic A-brane $\cu{B}_{cc}$ of $(X, \omega)$.

Associated with $\cu{B}_{cc, 0}$ is a non-formal deformation quantization constructed in \cite{ChaLeuLiSueYau2026}, which may be viewed as a mathematical realization of the endomorphism algebra of $\cu{B}_{cc, 0}$. In Subsection \ref{subsec:endomorphism algebra}, we incorporate the corresponding corrections into this deformation quantization, and propose that the resulting non-commutative algebra should be viewed as the endomorphism algebra of $\cu{B}_{cc}$, in accordance with our mirror theorem. In Subsection \ref{subsec:specialization of deformation quantization}, we discuss how this noncommutative algebra may be understood as a specialization of a deformation quantization.

\subsection{The semi-affine model}\label{subsec:construction of Bcc}

We begin with the following semi-affine space-filling A-brane
$$\cu{B}_{cc,0}:=(X_0,L_0,\nabla^{L_0})$$
on $X_0$, which extends the construction of \cite[Example 3.3.3]{ChaLeuLiSueYau2026} to the setting of multiple singular fibers. For an open subset $U \subset B^{sm}$, denote $p^{-1}(U)$ by $X_U$. For each chart $X_{U_i^{\pm}}$, $(L_0, \nabla^{L_0})$ is modelled by the $U(1)$-line bundle induced by the unitary automorphy factors
$$\tau(y,m):=\exp\left(4\pi\sqrt{-1}m^2y^1\right),$$
together with the $U(1)$-connection
\begin{equation}
  \label{eqn: local unitary connection}
  d-2\pi\sqrt{-1}\left(\frac{1}{2}x_1dx_2+2y^2dy^1\right).  
\end{equation}
The local models are glued by the gauge transformations
\begin{align*}
    g_{U_i^-U_i^+}(x_{U_i^-,2},y_{U_i^-}^1):=&\,\begin{dcases}
        1 & \text{ on }X_{B_i^+}\\
        \exp\left(2\pi\sqrt{-1}\left(\frac{1}{4}x_{U_i^-,2}^2-(y_{U_i^-}^1)^2\right)\right) & \text{ on }X_{B_i^-}
    \end{dcases},\\
    g_{U_{i+1}^-U_i^+}(x_{U_{i+1}^-},y_{U_{i+1}^-}):=&\,1,
\end{align*}
i.e. they satisfy the relation
\begin{align*}
    g_{U_i^-U_i^+}(x_{U_i^-,2},y_{U_i^-}^1+m_{U_i^-})\cdot\tau_{U_i^+}(y_{U_i^+}^1,m_{U_i^+})=&\,\tau_{U_i^-}(y_{U_i^-},m_{U_i^-})\cdot g_{U_i^-U_i^+}(x_{U_i^-,2},y_{U_i^-}^1),\\
    g_{U_{i+1}^-U_i^+}(x_{U_{i+1}^-,2},y_{U_{i+1}^-}^1+m_{U_{i+1}^-})\cdot\tau_{U_i^+}(y_{U_i^+}^1,m_{U_{i+1}^+})=&\,\tau_{U_{i+1}^-}(y_{U_{i+1}^-},m_{U_{i+1}^-})\cdot g_{U_{i+1}^-U_i^+}(x_{U_i^-,2},y_{U_i^-}^1),
\end{align*}
and identify the local unitary connections in \eqref{eqn: local unitary connection}.

Our goal is to modify the above gauge transformations as so to obtain a space-filling coisotropic A-brane of $(X, \omega)$. Since the existence of such a brane requires $X$ to admit a holomorphic symplectic structure, we first study the uncorrected complex structure $I_0:=\omega^{-1}F_{\nabla^{L_0}}$ on $X_0$ induced from the brane condition on $\cu{B}_{cc, 0}$. With respect to $I_0$, the manifold $X_0$ is equipped with the local complex coordinates
\begin{align*}
(z^1, z^2) :=\, (x_1-2\sqrt{-1}y^2, x_2+2\sqrt{-1}y^1).
\end{align*}
On $U_i^{\pm}$, the complex coordinates
\begin{align*}
    (u, w) := &\, (e^{\pi z^1}, e^{\pi z^2}).
\end{align*}
also receive monodromy induced by the affine change of coordinates, namely, around the affine singularity $s_i$, the counter-clockwise monodromy of these complex coordinates is given by
$$\begin{dcases}
    u_{U_i^-}\mapsto u_{U_i^-}w_{U_i^-}^{-1}\\
    u_{U_i^+}\mapsto u_{U_i^+}w_{U_i^-}^{-1}\\
    w_{U_i^{\pm}}\mapsto w_{U_i^{\pm}}.
\end{dcases}$$
Remarkably, this coincides precisely with the monodromy of the semi-flat mirror complex structure, as described in \eqref{eqn: monodromy on semi-flat mirror}! This motivates us to correct the gluing map on $X_{U_i^-\cap U_i^+}=X_{B_i^+}\sqcup X_{B_i^-}$ by
\begin{equation}\label{eqn: A gluing}
    \begin{dcases}
    u_{U_i^+}=u_{U_i^-} (1+w_{U_i^-}^{-1}),  & \text{ on } X_{B_i^+},\\
    u_{U_i^+}=u_{U_i^-} w_{U_i^-}^{-1}(1+w_{U_i^-}) & \text{ on } X_{B_i^-}\\
    w_{U_i^+}=w_{U_i^-}=w_{U_{i+1}^-}=w_{U_{i+1}^+} & \text{ for all }i\\
    u_{U_i^+}=u_{U_{i+1}^-} & \text{ for all }i,
\end{dcases}
\end{equation}
in analogy with the gluing rule \eqref{eqn: B gluing} on the B-side. We thus have constructed a complex structure $I$ on $X$. Denote the corresponding structure sheaf by $\cu{O}_X$. Note that $w:=w_{U_i^{\pm}}$ is a global function.

\begin{remark}
    The symplectic form $\omega$ is of type $(1, 1)$ with respect to the complex structure on $X$ inherited from $\mathbb{C}^2 \times \mathbb{C}^\times$, and of type $(2, 0) + (0, 2)$ with respect to $I$. Hence, these two complex structures are distinct. 
\end{remark}

\subsection{Construction of the space filling A-brane $\cu{B}_{cc}$}
\label{subsec:Construction of the space filling A-brane}
We now show that the local model of $(L_0,\nabla^{L_0})$ over $X_{U_i^{\pm}}$ can be glued to a $U(1)$-line bundle $(L,\nabla^L)\to X$ with respect to the corrected coordinate transformations in \eqref{eqn: A gluing}. Recall that $e^{2\pi\sqrt{-1}y^1}$ is a global function on $X$ and for each chart $X_{U_i^\pm}$, the unitary automorphy factor of $L_0$ is given by
\begin{equation}\label{eqn:factor of autopmorphy +}
    \tau(y,m) =\exp(4\pi\sqrt{-1}m^2y^1).
\end{equation}
Note that $U_i^-\cap U_i^+ = B_i^+\sqcup B_i^-$. In terms of the coordinates on $U_i^-$, we have
$$\begin{dcases}
    m_{U_i^+}^2=m_{U_i^-}^2 & \text{ on } X_{B_i^+},\\
    m_{U_i^+}^2=m_{U_i^-}^2+m_{U_i^-}^1 & \text{ on } X_{B_i^-},
\end{dcases}$$
and so, in terms of the coordinates on $X_{U_i^-}$, the automorphy factor of $L_0|_{X_{U_i^+}}$ is given by
$$\tau_{U_i^+}(y_{U_i^+}(y_{U_i^-}),m_{U_i^+}(m_{U_i^-}))=\begin{dcases}
    \exp(4\pi\sqrt{-1}m_{U_i^-}^2y^1) & \text{ on } X_{B_i^+},\\
    \exp(4\pi\sqrt{-1}(m_{U_i^-}^2+m_{U_i^-}^1)y^1) & \text{ on } X_{B_i^-}.
\end{dcases}$$
We notice that we can obtain a gauge transformation expressed in terms of the dilogarithm function
$$\Li_2(z):=-\int_0^z\frac{\log(1-t)}{t}dt.$$
To see this, recall the modified gluing on $X_{B_i^+}$:
$$\begin{dcases}
    x_{U_i^+,1}=x_{U_i^-,1}+\frac{1}{\pi}\log|1+w^{-1}|,\\
    y_{U_i^+}^2=y_{U_i^-}^2-\frac{1}{2\pi}\arg(1+w^{-1}).
\end{dcases}$$
It then follows from \eqref{eqn: local unitary connection} that
\begin{align*}
  \nabla^{L_0}|_{X_{U_i^+}}-\nabla^{L_0}|_{X_{U_i^-}}=& \frac{\sqrt{-1}}{\pi}\left(\log|1+w^{-1}|d\log|w^{-1}|-\arg(1+w^{-1})d\arg(w^{-1})\right)\\
  =& -\frac{\sqrt{-1}}{\pi}d\text{Re}\left(\Li_2(-w^{-1})\right),  
\end{align*}
where we used the identity
$$\frac{d}{dw}\Li_2(-w^{-1})=\frac{\log(1+w)-\log(w)}{w}.$$

On $X_{B_i^+}$, $\log|w|=\pi x_2>0$, which is equivalent to $|w|>1$. As $y^1\mapsto y^1+m^1$, the dilogarithm function receive no monodromy:
$$\Li_2(-w^{-1})\mapsto \Li_2(-w^{-1})$$
and thus defines a function on $X_{B_i^+}$. Similarly, on $X_{B_i^-}$, where $|w|<1$, we have
$$\frac{d}{dw}\Li_2(-w)=-\frac{\log(1+w)}{w}.$$
This implies
$$\Li_2(-w)\mapsto \Li_2(-w)$$
when $y^1\mapsto y^1+m^1$. Using this, we see that
$$G_{U_i^-U_i^+}(z^2):=\begin{dcases}
    \exp\left(-\frac{\sqrt{-1}}{\pi}\text{Re}\left(\Li_2(-w^{-1})\right)\right)  & \text{ on }X_{B_i^+},\\
    \exp\left(2\pi\sqrt{-1}\left(\frac{1}{4}x_{U_i^-,2}^2-(y_{U_i^-}^1)^2\right)\right)\exp\left(-\frac{\sqrt{-1}}{\pi}\text{Re}\left(\Li_2(-w)\right)\right) & \text{ on }X_{B_i^-},
\end{dcases}$$
gives a gauge equivalence between the automorphy factors of $L_0|_{X_{U_i^-}}$ and $L_0|_{X_{U_i^+}}$ on $X_{U_i^+\cap U_i^-}$, i.e.
$$G_{U_i^-U_i^+}(z^2+2\sqrt{-1}m^1)\cdot \tau_{U_i^+}(y_{U_i^+},m_{U_i^+})=\tau_{U_i^-}(y_{U_i^-},m_{U_i^-})\cdot G_{U_i^-U_i^+}(z^2).$$
A gauge equivalence between $(L_0|_{X_{U_i^+}},\nabla^{L_0}|_{X_{U_i^+}})$ and $(L_0|_{X_{U_{i+1}^-}},\nabla^{L_0}|_{X_{U_{i+1}^-}})$ trivially holds by taking $$G_{U_{i+1}^-U_i^+}\equiv 1 \quad \text{for all } i.$$


Thus, we successfully obtain a $U(1)$-bundle $(L,\nabla^L)\to X$. Its curvature is given by $-2\pi\sqrt{-1} F_{\nabla^L}$, where
$$F_{\nabla^L}:=\frac{1}{2}dx_1\wedge dx_2-2dy^1\wedge dy^2.$$
Together with the symplectic form $\omega$, the following defines a holomorphic symplectic form on $X$:
$$\Omega_X:=F_{\nabla^L}+\sqrt{-1}\omega=\frac{1}{2}d\left(x_1-2\sqrt{-1}y^2\right)\wedge d\left(x_2+2\sqrt{-1}y^1\right).$$
Hence, $\cu{B}_{cc}:=(X,L,\nabla^L)$ is a space-filling coisotropic $A$-brane of $(X, \omega)$. 
Consequently, for each $k \in \mathbb{Z}_{>0}$, $\cu{B}_{cc}^{(k)}:=(X,L^{\otimes k},\nabla^{L^{\otimes k}})$ is a space-filling coisotropic $A$-brane of $(X,k\omega)$.

\subsection{The endomorphism algebra of $\cu{B}_{cc}^{(k)}$}
\label{subsec:endomorphism algebra}

This subsection is devoted to a proposed construction of the endomorphism algebra $\Hom_A(\cu{B}_{cc}^{(k)},\cu{B}_{cc}^{(k)})$ for each $k \in \mathbb{Z}_{>0}$. In the previous joint work \cite{ChaLeuLiSueYau2026}, this algebra was realized as a non-formal holomorphic deformation quantization of $(X, \Omega_X)$ in the semi-flat setting, obtained by gluing analytic quantum tori. The construction presented here follows the same spirit, with wall-crossing factors incorporated into the gluing data.

We introduce the new parameter $q:=e^{-\frac{\pi\sqrt{-1}}{k}}$ and for any rational number $r\in\bb{Q}$, we write $q^r:=e^{-\frac{r\pi\sqrt{-1}}{k}}$. Note that $q$ is a primitive $2k$-th root of unity, where $2k$ is precisely the rank of the mirror B-brane of $\cu{B}_{cc}^{(k)}$, to be constructed in the next section. Define the skew-symmetric pairing $\{-,-\}:\bb{Z}^2\times\bb{Z}^2\to\bb{Q}$ by
$$\{m,m'\}:=\frac{1}{2}(m^1m'^2-m'^1m^2),$$
where we write $m = (m^1, m^2)$ and $m' = (m'^1, m'^2)$. 
For holomorphic functions $f,g\in H^0(X_{U_i^{\pm}},\cu{O}_X)$, we expand them as Fourier series
\begin{equation}
    \label{eqn: Fourier expansion}
    f(u, w) = \sum_{m \in \mathbb{Z}^{2n}} \widehat{f}_m u^{m^1} w^{m^2}, \quad g(u, w) = \sum_{m \in \mathbb{Z}^{2n}} \widehat{g}_m u^{m^1} w^{m^2}
\end{equation}
and define
\begin{equation*}
    f\star_{k^{-1}}g:=\sum_{m,m'\in\bb{Z}^2} q^{\{m,m'\}}  \widehat{f}_m\widehat{g}_{m'} u^{m^1+m'^1} w^{m^2+m'^2}.
\end{equation*}
This product gives a non-commutative algebra $(H^0(X_{U_i^{\pm}},\cu{O}_X),\star_{k^{-1}})$, which is an analytic quantum torus. We glue these quantum tori together via the following gluing formulae
\begin{equation}\label{eqn:quantum wall crossing}
    \begin{dcases}
    u_{U_i^+}=u_{U_i^-} (1+w^{-2k})^{\frac{1}{2k}} & \text{ on } X_{B_i^+},\\
    u_{U_i^+}=u_{U_i^-} w^{-1}(1+w^{2k})^{\frac{1}{2k}} & \text{ on } X_{B_i^-},\\
    w_{U_i^+}=w_{U_i^-}=w_{U_{i+1}^-}=w_{U_{i+1}^+} & \text{ for all }i,\\
    u_{U_i^+}=u_{U_{i+1}^-} & \text{ for all }i.
\end{dcases}
\end{equation}
Here, we have chosen the principal branch of the $2k$-th root function, i.e. the branch for which $1^{\frac{1}{2k}}=1$. Since
$$\begin{dcases}
    \log(1+w^{-2k})\mapsto\log(1+w^{-2k}) & \text{ on }X_{B_i^+},\\
    \log(1+w^{2k})\mapsto\log(1+w^{2k}) & \text{ on }X_{B_i^-},
\end{dcases}$$
as $y^1\mapsto y^1+m^1$, both factors $(1+w^{\pm 2k})^{\frac{1}{2k}}$ are well-defined functions. Note that on the overlap $U_i^-\cap U_i^+$, $\log|w|=2\pi x_2\neq 0$, so $|w|\neq 1$, and hence the factors $(1+w^{\pm 2k})^{\frac{1}{2k}}$ are invertible. Moreover,
$$u_{U_i^+}=u_{U_i^-} w^{-1}(1+w^{2k})^{\frac{1}{2k}}=u_{U_i^-} (1+w^{-2k})^{\frac{1}{2k}}.$$
Thus \eqref{eqn: A gluing} gives us a sheaf of monodromy-free algebra $\cu{A}_X^{(k)}$ over $B$ so that $H^0(U_i^{\pm},\cu{A}_X^{(k)})=H^0(X_{U_i^{\pm}},\cu{O}_X)$.

\begin{remark}
    Any holomorphic function whose nonzero Fourier modes in the expansion \eqref{eqn: Fourier expansion} are all multiples of $2k$ lies in the center of the algebra $(H^0(X_{U_i^{\pm}}, \mathcal{O}_X), \star_{k^{-1}})$. In particular, $(1+w^{\pm 2k})^{\frac{1}{2k}}$ is a central element of the quantum tori, so it makes no difference whether one writes $u_{U_i^-} (1+w^{-2k})^{\frac{1}{2k}}$ or $u_{U_i^-} \star_{k^{-1}}(1+w^{-2k})^{\frac{1}{2k}}$.
\end{remark}

If we choose a different branches of the $2k$-th root function, they differ by a multiple of a $2k$-root of unity $\zeta$. The modified gluing is defined to be
\begin{equation}\label{eqn:quantum wall crossing2}
    \begin{dcases}
    u_{U_i^+}=u_{U_i^-} \zeta(1+w^{-2k})^{\frac{1}{2k}} & \text{ on } X_{B_i^+},\\
    u_{U_i^+}=u_{U_i^-} w^{-1}\zeta(1+w^{2k})^{\frac{1}{2k}} & \text{ on } X_{B_i^-},\\
    w_{U_i^+}=w_{U_i^-}=w_{U_{i+1}^-}=w_{U_{i+1}^+} & \text{ for all }i,\\
    u_{U_i^+}=\zeta u_{U_{i+1}^-} & \text{ for all }i.
\end{dcases}
\end{equation}
Then the maps $\varphi_{U_i^{\pm}}$ given by
$$\varphi_{U_i^+}:\begin{dcases}
    u_{U_i^+}\mapsto u_{U_i^+},\\
    w\mapsto w,
\end{dcases}\, \quad \text{and} \quad \,
\varphi_{U_i^-}:\begin{dcases}
    u_{U_i^-}\mapsto\zeta^{-1} u_{U_i^-},\\
    w\mapsto w,
\end{dcases}$$
provide an isomorphism between the sheaves of algebras that obtained from the gluing \eqref{eqn:quantum wall crossing}, \eqref{eqn:quantum wall crossing2}. Thus, the sheaf of algebras $\cu{A}_X^{(k)}$ is independent of the choice of branch of the $2k$-th root function. We propose that
$$\Hom_A(\cu{B}_{cc}^{(k)},\cu{B}_{cc}^{(k)}):=(H^{\bullet}(B,\cu{A}_X^{(k)}),\star_{k^{-1}}).$$

\begin{remark}
    A notable feature of the wall-crossing factors is the appearance of the parameter $2k$. Since $q$ is a $2k$-th root of unity, the $q$-quantum torus may be viewed heuristically as a $2k$-th root of the Laurent polynomial ring (the classical torus). In the same spirit, the gluing \eqref{eqn:quantum wall crossing} may be regarded as a $2k$-th root of the classical gluing \eqref{eqn: A gluing}. At present, however, this heuristic interpretation lacks a purely symplectic-geometric explanation, and its precise geometric significance remains an interesting direction for future investigation.
\end{remark}

\begin{remark}
\label{remark: limit of A_X}
As $k\to\infty$, we have $(1+w^{\pm 2k})^{\frac{1}{2k}} \to 1$. Consequently, the sheaf $\cu{A}_X^{(k)}$ degenerates to $\iota_*p_{0*}\cu{O}_{X_0}$ rather than $p_*\cu{O}_X$, where $\cu{O}_{X_0}$ denotes the structure sheaf of $(X_0,I_0)$ and $\iota:B^{\mathrm{sm}}\hookrightarrow B$ is the inclusion. This is because the quantum parameter $\frac{1}{k}$ corresponds to the rescaled symplectic form $k\omega$. Thus, the limit $k\to\infty$ is precisely the large-volume limit, in which the symplectic area of every non-constant holomorphic disc tends to infinity. Consequently, all instanton corrections are exponentially suppressed, and the instanton-corrected complex structure $I$ reduces to its semi-flat counterpart in the limit.
\end{remark}

Let $\iota:B^{sm}\hookrightarrow B$ be the inclusion map. By the construction of $\cu{A}_X^{(k)}$, we have

\begin{lemma}\label{lem:A isomorphism}
    We have $H^q(B^{sm},\iota^{-1}\cu{A}_X^{(k)})=H^q(B,\cu{A}_X^{(k)})$, for all $q\geq 0$.
\end{lemma}
\begin{proof}
    This is because a $T_\mathbb{Z}^* B^{sm}$-invariant local holomorphic function on $X_0$ that respects the gluing \eqref{eqn:quantum wall crossing} extends uniquely to a local $T_\mathbb{Z}^* B^{sm}$-invariant section of $\cu{A}_X^{(k)}$.
\end{proof}

This allows us to compute $H^{\bullet}(B,\cu{A}_X^{(k)})$ by C\v{e}ch cohomology. More precisely, we refine the open sets $U_i^{\pm}$ by decomposing each of them into a union of three convex open subsets $U_{i,j}^{\pm}$, $j=1,2,3$, with the corresponding gluing maps as shown in Figure \ref{fig:refined_charts}. Then $\{U_{i,j}^{\pm}\}$ forms an acyclic cover for $\iota^{-1}\cu{A}_X^{(k)}$ because $X_{U_{i,j}^{\pm}}$ and their intersections are all Stein. Hence
$$H^q(B,\cu{A}_X^{(k)})=\check{H}^q(\{U_{i,j}^{\pm}\},\iota^{-1}\cu{A}_X^{(k)}).$$
This will be useful in proving our mirror theorem.
\begin{figure}[H]
    \centering
    \includegraphics[width=0.6\linewidth]{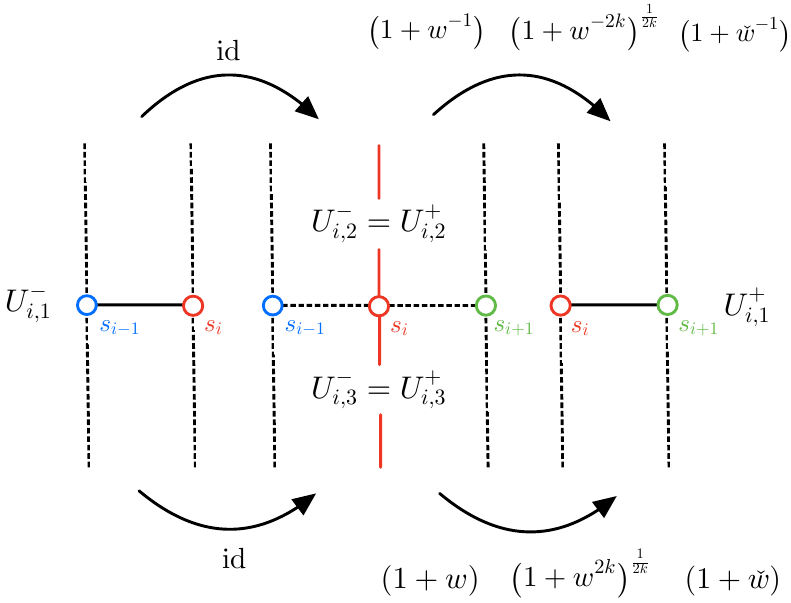}
    \caption{Refining $U_i^{\pm}$ into three convex open subsets and their corresponding gluing maps.}
    \label{fig:refined_charts}
\end{figure}

\subsection{$\cu{A}_X^{(k)}$ as a specialization of deformation quantization}
\label{subsec:specialization of deformation quantization}

As observed in Remark \ref{remark: limit of A_X}, the sheaf $\cu{A}_X^{(k)}$ does not asymptotically recover $p_*\cu{O}_X$ as $k \to \infty$. Consequently, it does not immediately resemble a deformation quantization of $\cu{O}_X$. Moreover, the appearance of the factor $(1+w^{\pm 2k})^{\frac{1}{2k}}$ prevents one from simply replacing $\frac{1}{k}$ by a formal parameter to obtain a deformation quantization in the manner of \cite{ChaLeuLiSueYau2026}. Nevertheless, in this subsection we show that, for each level $k\in\bb{Z}_{>0}$, there exists a strict deformation quantization $\cu{A}_{X,\hbar}^{(k)}$ of $(X,\Omega_X)$ in the sense of Definition \ref{def:relative DQ} satisfying $\cu{A}_{X,-\frac{\pi}{k}}^{(k)}=\cu{A}_X^{(k)}$, i.e. $\cu{A}_X^{(k)}$ is a specialization of a deformation quantization at $\hbar=-\frac{\pi}{k}$.

Since we are dealing with ``semi-local" functions in the sense that they are defined on $X_U\subset X$, it is natural to formulate deformation quantization relative to the base $B$.

\begin{definition}\label{def:relative DQ}
    Let $(X,\Omega_X)$ be a holomorphic symplectic manifold and suppose $p:X\to B$ be a continuous map. A sheaf of \emph{deformation quantizations of $(X,\Omega_X)$ relative to $B$} is a sheaf of $\bb{C}[[\hbar]]$-algebras $\cu{A}_{X,\hbar}$ over $B$ together with a surjective algebra homomorphism $\cu{A}_{X,\hbar}\to p_*\cu{O}_X$ such that for any small enough open set $U\subset B$, the algebra homomorphism $\cu{A}_{X,\hbar}(U)\to\cu{O}_X(X_U)$ defines a deformation quantization of $(X_U,\Omega_X|_{X_U})$. A sheaf of \emph{deformation quantizations of $(X,\Omega_X)$ relative to $B$} is called \emph{strict} if for any open set $U\subset B$, $\cu{A}_{X,\hbar}(U)\to\cu{O}_X(X_U)$ is a strict quantization.
\end{definition}

Let $\ul{q}:=e^{\sqrt{-1}\hbar}$. We construct $\cu{A}_{X,\hbar}^{(k)}$ via the following gluing of $\ul{q}$-quantum tori:
\begin{equation}\label{eqn:q-dilog gluing}
    \begin{dcases}
    u_{U_i^+}=u_{U_i^-}\star_{\ul{q}}(1+\ul{q}^{\frac{1}{2}}w^{-1})\frac{h_+(\ul{q}w^{-1})}{h_+(w^{-1})} & \text{ on } X_{B_i^+},\\
    u_{U_i^+}=u_{U_i^-}w^{-1}\star_{\ul{q}}(1+\ul{q}^{-\frac{1}{2}}w)\frac{h_-(\ul{q}^{-1}w)}{h_-(w)} & \text{ on } X_{B_i^-},\\
    w_{U_i^+}=w_{U_i^-}=w_{U_{i+1}^-}=w_{U_{i+1}^+} & \text{ for all }i,\\
    u_{U_i^+}=u_{U_{i+1}^-} & \text{ for all }i,
\end{dcases}
\end{equation}
where
$$h_{\pm}(w):=\prod_{r=0}^{2k-1}(1+\ul{q}^{\pm(\frac{1}{2}+r)}w)^{-\frac{r}{2k}}.$$
One can easily check that \eqref{eqn:q-dilog gluing} gives a monodromy free sheaf of algebras over $B$. Clearly, when $\hbar=0$, equivalently, $\ul{q}=1$, we get back the gluing \eqref{eqn: A gluing} of $\cu{O}_X$. If we treat $\hbar$ as a formal parameter, modulo $\hbar$ gives a surjective map
$$\cu{A}_{X,\hbar}^{(k)}\to p_*\cu{O}_X$$
so that $\cu{A}_{X,\hbar}^{(k)}/\hbar\cu{A}_{X,\hbar}^{(k)}\cong p_*\cu{O}_X$. Thus $\cu{A}_{X,\hbar}^{(k)}$ is a sheaf of relative deformation quantizations of $(X,\Omega_X)$ relative to $B$ for every $k\in\bb{Z}_{>0}$.

The function $h$ is defined so that we have the following

\begin{lemma}\label{lem:local deformation}
    At $\ul{q}=q=e^{-\frac{\pi\sqrt{-1}}{k}}$, we have
    \begin{align*}
        (1+q^{\frac{1}{2}}w^{-1})\frac{h_+(qw^{-1})}{h_+(w^{-1})}=&\,(1+w^{-2k})^{-\frac{1}{2k}},\\
        (1+q^{-\frac{1}{2}}w)\frac{h_-(q^{-1}w)}{h_-(w)}=&\,(1+w^{2k})^{-\frac{1}{2k}}.
    \end{align*}
\end{lemma}
\begin{proof}
    Using $q^{2k}=1$, we have
    \begin{align*}
        h_+(qw^{-1})=&\,\prod_{r=0}^{2k-1}(1+q^{\frac{1}{2}+r}qw^{-1})^{-\frac{r}{2k}}\\
        =&\,\prod_{r=0}^{2k-1}(1+q^{\frac{1}{2}+r+1}w^{-1})^{-\frac{r+1}{2k}}(1+q^{\frac{1}{2}+r+1}w^{-1})^{\frac{1}{2k}}\\
        =&\,(1+q^{\frac{1}{2}}w^{-1})^{-1}\prod_{r=0}^{2k-2}(1+q^{\frac{1}{2}+r+1}w^{-1})^{-\frac{r+1}{2k}}\prod_{r=0}^{2k-1}(1+q^{\frac{1}{2}+r+1}w)^{\frac{1}{2k}}\\
        =&\,(1+q^{\frac{1}{2}}w^{-1})^{-1}\prod_{r=0}^{2k-1}(1+q^{\frac{1}{2}+r}w^{-1})^{-\frac{r}{2k}}\prod_{r=0}^{2k-1}(1+q^{\frac{1}{2}+r}w^{-1})^{\frac{1}{2k}}\\
        =&\,(1+q^{\frac{1}{2}}w^{-1})^{-1}h_+(w)(1+w^{-2k})^{\frac{1}{2k}},
    \end{align*}
    where, in the last equality, we used the factorization
    $$x^{2k}+w^{-2k}=\prod_{r=0}^{2k-1}(x-q^{\frac{1}{2}+r}w^{-1})$$
    and evaluated at $x=-1$. The identity for $h_-$ can be proved similarly.
\end{proof}

By Lemma \ref{lem:local deformation}, we obtain

\begin{proposition}\label{prop:A as DQ}
    For each $k\in\bb{Z}_{>0}$, there is a sheaf $\cu{A}_{X,\hbar}^{(k)}$ of strict deformation quantizations of $(X,\Omega_X)$ relative to $B$ such that, $\cu{A}_{X,-\frac{\pi}{k}}^{(k)}=\cu{A}_X^{(k)}$.
\end{proposition}

\begin{remark}
    Treating $\hbar$ as a formal parameter, the quantum dialogarithm
    $$H(w):=\sum_{r=1}^{\infty}\frac{(-1)^{r-1}}{r}\frac{w^r}{\ul{q}^{\frac{r}{2}}-\ul{q}^{-\frac{r}{2}}}$$
    is defined, and we can rewrite the first two gluing equations of \eqref{eqn:q-dilog gluing} as an inner automorphism (cf. \cite[Lemma 3.3, 3.4]{Bou2020a}):
    $$u_{U_i^+}=\begin{dcases}
        h_+(w^{-1})\exp(H(w^{-1}))\star_{\ul{q}} u_{U_i^-}\star_{\ul{q}}\exp(-H(w^{-1}))h_+(w^{-1})^{-1} & \text{ on }X_{B_i^+}\\
        h_-(w)\exp(H(w))\star_{\ul{q}} u_{U_i^-}w^{-1}\star_{\ul{q}}\exp(-H(w))h_-(w)^{-1} & \text{ on }X_{B_i^-}
    \end{dcases}.$$
    In \cite{Bou2020a, Bou2020b}, Bousseau quantized the structure sheaf of log Calabi--Yau surfaces obtained by Gross--Hacking--Keel \cite{GroHacKee2015} using almost exactly the same inner automorphism (without the twisting term $h(w)$) and showed that the quantum dilogarithm is basically the generating function of all genus (log) Gromov--Witten invariants of the mirror Looijenga pair. We may think of our algebra $\cu{A}_X^{(k)}$ as a twisting of the algebra obtained by Bousseau.
\end{remark}

\section{The mirror B-brane of $\cu{B}_{cc}^{(k)}$}\label{sec:mirror B brane}

Parallel to the previous section, where we constructed a space-filling coisotorpic A-brane $\cu{B}_{cc}^{(k)}$, we now turn to the construction of its mirror B-brane. In Subsection \ref{subsec:The semi-affine B-brane}, we review the local construction of the mirror B-brane $\check{\cu{B}}_{cc,0}^{(k)}$ associated with the semi-affine coisotropic A-brane $\cu{B}_{cc, 0}^{(k)}$. This construction was first introduced in \cite{ChaLeuZha2018}; see also \cite{ChaLeuLiSueYau2026}. In Subsection \ref{subsec:The instanton corrected B-brane}, we introduce an \emph{instanton-corrected} version of $\check{\cu{B}}_{cc,0}^{(k)}$. The resulting object, denoted by $\check{\cu{B}}_{cc}^{(k)}$, will be referred to as the \emph{mirror B-brane} of $\cu{B}_{cc}^{(k)}$.

\subsection{The semi-affine mirror B-brane}
\label{subsec:The semi-affine B-brane}
Fix a level $k \in \mathbb{Z}_{>0}$. Applying Chan--Leung--Zhang's transform \cite{ChaLeuZha2018} to $\cu{B}_{cc,0}^{(k)}:=(X_0,L_0^{\otimes k},\nabla^{L_0^{\otimes k}})$, we obtain its mirorr B-brane $\check{\cu{B}}_{cc,0}^{(k)}=(\check{X}^{(k)}_0,\check{E}_0^{(k)})$. The vector bundle $\check{E}_0^{(k)}$ is not only holomorphic, but is also equipped with a Hermitian metric so that the triple $(\check{X}^{(k)}_0,\check{E}_0^{(k)},\nabla^{\check{E}_0^{(k)}})$ is semi-affine in the sense of \cite{ChaLeuZha2018}, where $\nabla^{\check{E}_0^{(k)}}$ denotes the Chern connection of $\check{E}_0^{(k)}$. In particular, the holomorphic structure on $\check{E}_0^{(k)}$ is determined by the $(0, 1)$-part of $\nabla^{\check{E}_0^{(k)}}$.

In what follows, to simplify the notation, we shall suppress the superscript $(k)$ whenever the dependence on the level is clear from the context, reinstating it only when it plays an essential role, particularly in the statements of the main propositions and theorems.

Recall that, on each chart $\check{X}_{U_i^{\pm}}$, the local construction of $(\check{E}_0,\nabla^{\check{E}_0})$ via the kernel of family Dirac operators in \cite{ChaLeuZha2018} is naturally isomorphic to the construction via family geometric quantization developed in \cite{ChaLeuLiSueYau2026}. Both approaches provide concrete local models and explicit coordinate descriptions of the mirror bundle. However, neither yields an explicit description in terms of automorphy factors.

In the present work, we adopt a third approach and describe the local mirror bundle directly through its automorphy factors. This viewpoint also appears in Sis\'ak's thesis \cite{Sis2024}. The resulting description gives an explicit realization of the mirror bundle and its transition data in terms of finite-dimensional representations of the Heisenberg group. To this end, recall that we introduced the primitive $2k$-th root of unity $q := e^{-\frac{\pi\sqrt{-1}}{k}}$ and set $q^r := e^{-\frac{r\pi\sqrt{-1}}{k}}$ for all $r \in \mathbb{Q}$. Consider the $2k \times 2k$ matrices
$$C:=\begin{pmatrix}
1 & 0 & 0 & \cdots & 0\\
0 & q^{-1} & 0 & \cdots & 0\\
0 & 0 & q^{-2} & \cdots & 0\\
\vdots & \vdots & \vdots & \ddots & \vdots\\
0 & 0 & 0 & \cdots & q^{-(2k-1)} 
\end{pmatrix},\quad S:=
\begin{pmatrix}
0 & 0 & 0 & \cdots & 0 & 1\\
1 & 0 & 0 & \cdots & 0 & 0\\
0 & 1 & 0 & \cdots & 0 & 0\\
\vdots & \vdots & \ddots & \vdots & \vdots\\
0 & 0 & 0 & \cdots & 0 & 0\\
0 & 0 & 0 & \cdots & 1 & 0
\end{pmatrix}.$$
These are commonly referred to as the \emph{clock matrix} and \emph{shift matrix}, respectively (cf. \cite{AscDes2019, DanFioFra2014}). 
They satisfy the fundamental commutation relation
$$SC=qCS,$$
and generate the standard $2k$-dimensional representation of the Heisenberg group over $\bb{Z}_{2k} := \mathbb{Z} / (2k\mathbb{Z})$.

Now, we describe the holomorphic vector bundle $\check{E}_0$ together with the connection $\nabla^{\check{E}_0}$ locally over the open subsets $\check{X}_{U_i^{\pm}}$. As will become clear in the next subsection, the Hermitian metric on $E_0$ plays no role in our considerations and will therefore be omitted from the description below. A proof that the description given below is equivalent to the local construction in \cite{ChaLeuLiSueYau2026} will be provided in Appendix \ref{app:Toeplitz}.

We first pass to the universal cover $TU_i^{\pm}$ of $\check{X}_{U_i^{\pm}}$. On $TU_i^{\pm}$, the pullback of the vector bundle $\check{E}_0$ is trivial of rank $2k$. Choosing a smooth frame $\{1^{(\alpha)}\}_{\alpha=1}^{2k}$ over $TU_i^\pm$, the pullback of the connection $\nabla^{\check{E}_0}$ is given by
$$d-\pi\sqrt{-1}\left(kx_1dx_2+\frac{1}{k}\check{y}_1d\check{y}_2\right)\Id_{2k}.$$
To recover $(\check{E}_0, \nabla^{\check{E}_0})$ locally over $\check{X}_{U_i^{\pm}}$, we equip this trivial bundle with the automorphy factors
\begin{equation}
    \label{eqn: unitary automorphy factor}
    \check{\tau}^{\text{uni}}(x, \check{y},m):=\exp\left(\frac{\pi\sqrt{-1}}{k}m_1\check{y}_2\right)A_m^{\text{uni}}, \quad m = (m_1, m_2),
\end{equation}
which determine the descent from $TU_i^{\pm}$ to $\check{X}_{U_i^{\pm}}$. Here, $\{A_m^{\text{uni}}\}_{m\in\bb{Z}^2}$ is a family of $2k\times 2k$ matrices defined by
$$A_m^{\text{uni}}:=q^{-m_1m_2}S^{m_1}C^{m_2},$$
for $m=(m_1,m_2)$. The superscript ``$\text{uni}$'' is used to distinguish this construction from the one described later in terms of holomorphic frames. Using the relation $SC=qCS$,
it is easy to see that $\{A_m^{\text{uni}}\}_{m\in\bb{Z}^2}$ satisfies
$$A_{m+m'}^{\text{uni}}=q^{-m_1'm_2}A_{m'}^{\text{uni}}A_m^{\text{uni}},$$
which is equivalent to the cocycle condition
$$\check{\tau}^{\text{uni}}(x, \check{y},m+m')=\check{\tau}^{\text{uni}}\left(x, \check{y}+m,m'\right) \cdot \check{\tau}^{\text{uni}}(x, \check{y},m).$$
Since the above description agrees with the local construction in \cite{ChaLeuLiSueYau2026}, there exist distinguished gauge transformations that glue the local data into the global mirror bundle $(\check{E}_0,\nabla^{\check{E}_0})$. As an explicit description of these gauge transformations will not be needed in this article, we omit it.


\subsection{The instanton corrected mirorr B-brane}
\label{subsec:The instanton corrected B-brane}
We now incorporate wall-crossing corrections into the construction of the mirror B-brane. A natural first attempt is to modify the gluing of the local models of $\nabla^{\check{E}_0}$ according to the corrected coordinate transformation
$$\begin{dcases}
    x_{U_i^+,1}=x_{U_i^-,1}+\frac{1}{2\pi k}\log|1+\check{w}^{\pm 1}|,\\
    \check{y}_{U_i^+,1}=\check{y}_{U_i^-,1}+\frac{1}{2\pi}\arg(1+\check{w}^{\pm 1}),
\end{dcases}$$
with the aim of producing a monodromy-free unitary connection. However, a straightforward calculation shows that
\begin{align*}
\nabla^{\check{E}_0}|_{\check{X}_{U_i^+}}-\nabla^{\check{E}_0}|_{\check{X}_{U_i^-}}=&\,\frac{\sqrt{-1}}{4\pi k}\left(\log|1+\check{w}^{-1}|d\log|\check{w}^{-1}|+\arg(1+\check{w}^{-1})d\arg(\check{w}^{-1})\right)\\
    =&\,\frac{\sqrt{-1}}{4\pi k}\text{Re}\left(\log(1+\check{w}^{-1})d\log(\ol{\check{w}}^{-1})\right),
\end{align*}
which is not $d$-closed, so the connection $\nabla^{\check{E}_0}$ cannot be extended to $\check{X}$. Nevertheless, the holomorphic strtucture can be glued. To see this, one finds that
$$(\nabla^{\check{E}_0})^{(0,1)}=\dbar-\frac{k\pi\sqrt{-1}}{2}\check{z}_1 d\ol{\check{z}_2} \cdot \Id_{2k} \quad \text{on } \check{X}_{U_i^{\pm}}$$
and
$$e^{(\alpha)}:=\exp\left(\frac{k\pi\sqrt{-1}}{2}\check{z}_1\ol{\check{z}}_2\right)1^{(\alpha)}, \quad \alpha = 1, \dots, 2k,$$
give a holomorphic frame of the pullback of $\check{E}_0$ over $TU_i^\pm$. In this holomorphic frame, the automorphy factors are given by
\begin{equation}
    \label{eqn: holomorphic automorphy factor}
    \check{\tau}(\check{z},m):=\exp\left(\frac{\pi}{2}(m_1\check{z}_2-m_2\check{z}_1)\right)A_m,
\end{equation}
where
$$A_m:=q^{\frac{1}{2}m_1m_2}A_m^{\text{uni}}=q^{-\frac{1}{2}m_1m_2}S^{m_1}C^{m_2}.$$
It satisfies
$$A_{m+m'}=q^{\frac{1}{2}(m_2'm_1-m_1'm_2)}A_{m'}A_m,$$
which is equivalent to the cocycle condition
$$\check{\tau}(\check{z},m+m')=\check{\tau} \left(\check{z}+\frac{\sqrt{-1}}{k}m,m'\right) \cdot \check{\tau}(\check{z},m).$$

Next, we define gauge transformations relating the local models of $\check{E}_0$ introduced above. Note that on the overlap $U_i^-\cap U_i^+=B_i^+\cup B_i^-$, we have
$$\begin{dcases}
    m_{U_i^+,1}=m_{U_i^-,1} & \text{ on } \check{X}_{B_i^+},\\
    m_{U_i^+,1}=m_{U_i^-,1}-m_{U_i^-,2} & \text{ on } \check{X}_{B_i^-},
\end{dcases}$$
This gives
$$A_{m_{U_i^+}}=\begin{dcases}
    A_{m_{U_i^-}}  & \text{ on } \check{X}_{B_i^+},\\
    q^{\frac{1}{2}m_2^2}S^{-m_2}A_{m_{U_i^-}} & \text{ on } \check{X}_{B_i^-}.
\end{dcases}$$
Therefore, in terms of the coordinates on $\check{X}_{U_i^-}$, we have
\begin{align*}
    & \check{\tau}_{U_i^+}(\check{z}_{U_i^+}(\check{z}_{U_i^-}),m_{U_i^+}(m_{U_i^-}))\\
    =& \begin{dcases}
    \exp\left(\frac{\pi}{2}(m_{U_i^-,1}\check{z}_2-m_2\check{z}_{U_i^-,1})\right)\exp\left(-\frac{m_2}{4k}\log(1+\check{w}^{-1})\right)A_{m_{U_i^-}} & \text{ on }\check{X}_{B_i^+},\\
    \exp\left(\frac{\pi}{2}(m_{U_i^-,1}\check{z}_2-m_2\check{z}_{U_i^-,1})\right)\exp\left(-\frac{m_2}{4k}\log(1+\check{w})\right)q^{\frac{1}{2}m_2^2}S^{-m_2}A_{m_{U_i^-}} & \text{ on }\check{X}_{B_i^-}.
\end{dcases}
\end{align*}
Let
\begin{equation}
    \label{eqn: definition of T}
    T:=\sum_{r=0}^{2k-1}q^{-\frac{1}{2}r^2}S^r.
\end{equation}
Using the fact that $q$ is a $2k$-th root of unity, a straightforward calculation shows that
\begin{equation}\label{eqn:commuting T and A}
    Tq^{\frac{1}{2}m_2^2}S^{-m_2}A_{m_{U_i^-}}=A_mT.
\end{equation}
Define
$$\check{G}_{U_i^-U_i^+}(\check{z}_2):=\begin{dcases}
    \exp\left(-\frac{\sqrt{-1}}{4}\check{z}_2\log(1+\check{w}^{-1})\right)\Id_{2k} & \text{ on }\check{X}_{B_i^+},\\
    \exp\left(-\frac{\sqrt{-1}}{4}\check{z}_2\log(1+\check{w})\right)T & \text{ on }\check{X}_{B_i^-},
\end{dcases}$$
Using \eqref{eqn:commuting T and A} and the fact that
$$\begin{dcases}
    \log(1+\check{w}^{-1})\mapsto\log(1+\check{w}^{-1}) & \text{ if }|\check{w}|>1,\\
    \log(1+\check{w})\mapsto\log(1+\check{w}) & \text{ if }|\check{w}|<1  
\end{dcases}$$
when $\check{y}_2\mapsto\check{y}_2+m_2$, 
we see that $\check{G}_{U_i^-U_i^+}(\check{z}_2)$ is a well-defined gauge equivalence between the automorphy factors on $\check{X}_{U_i^+\cap U_i^-}$, i.e.
$$\check{G}_{U_i^-U_i^+}\left(\check{z}_2+\frac{\sqrt{-1}}{k}m_2\right)\cdot \check{\tau}_{U_i^+}\left(\check{z}_{U_i^+}(\check{z}_{U_i^-}),m_{U_i^+}(m_{U_i^-})\right)=\check{\tau}_{U_i^-}(\check{z}_{U_i^-},m_{U_i^-})\cdot \check{G}_{U_i^-U_i^+}\left(\check{z}_2\right).$$
The gauge equivalence between $\check{E}_0|_{\check{X}_{U_i^+}}$ and $\check{E}_0|_{\check{X}_{U_{i+1}^-}}$ is again trivial by taking $\check{G}_{U_{i+1}^-U_i^+}\equiv 1$.

Note that the above treatments apply to each level $k$. This finishes the construction of a holomorphic vector bundle $\check{E}^{(k)}$ of rank $2k$ over $\check{X}^{(k)}$. We call $\check{\cu{B}}_{cc}^{(k)} := (\check{X}^{(k)},\check{E}^{(k)})$ the \emph{mirror B-brane} of $\cu{B}_{cc}^{(k)}$.

Following the standard description of the category of B-branes on $\check{X}^{(k)}$ in terms of the derived category of coherent sheaves, we define the graded algebra
$$\Hom_B(\check{\cu{B}}_{cc}^{(k)},\check{\cu{B}}_{cc}^{(k)}):=(H^{\bullet}(\check{X}^{(k)},\End(\check{E}^{(k)})),\circ),$$
where the product $\circ$ is induced by composition of endomorphisms and the cup product on cohomology.

In analogous to Lemma \ref{lem:A isomorphism}, we have

\begin{lemma}\label{lem:B isomorphism}
    We have $H^q(\check{X}^{(k)'},\End(E^{(k)})|_{\check{X}^{(k)'}})=H^q(\check{X}^{(k)},\End(\check{E}^{(k)}))$, for all $q\geq 0$, where $\check{X}^{(k)'}$ is the complex manifold whose underlying topological space is $\check{X}^{(k)}_0$ with gluing given by \eqref{eqn: B gluing}.
\end{lemma}

\section{The mirror theorem}\label{sec:mirror theorem}

In the semi-flat setting of \cite{ChaLeuLiSueYau2026}, for each level $k \in \mathbb{Z}_{>0}$, a graded algebra isomorphism
\begin{equation}
    \label{eqn:semi-flat mirror transform}
    \operatorname{Hom}_A(\cu{B}_{cc, 0}^{(k)}, \cu{B}_{cc, 0}^{(k)}) \cong \Hom_B(\check{\cu{B}}_{cc, 0}^{(k)},\check{\cu{B}}_{cc, 0}^{(k)})
\end{equation}
is constructed via twisted family Toeplitz operators, estabilshing a mirror theorem for coisotropic A-branes. Here $\operatorname{Hom}_A(\cu{B}_{cc, 0}^{(k)}, \cu{B}_{cc, 0}^{(k)}) := (H^\bullet(X_0, \mathcal{O}_{X_0}), \star_{0, k^{-1}})$ for a canonical non-commutative product $\star_{0, k^{-1}}$ and $\Hom_B(\check{\cu{B}}_{cc, 0}^{(k)},\check{\cu{B}}_{cc, 0}^{(k)}) := (H^\bullet(\check{X}_0^{(k)}, \operatorname{End} (\check{E}_0^{(k)})), \circ)$.

In this section, we construct a correction of \eqref{eqn: semi-flat mirror transform}:
\begin{equation*}
        \Phi^{(k)}: \Hom_A(\cu{B}_{cc}^{(k)},\cu{B}_{cc}^{(k)}) \to \Hom_B(\check{\cu{B}}_{cc}^{(k)},\check{\cu{B}}_{cc}^{(k)}),
\end{equation*}
which we call the \emph{instanton-corrected mirror transform}, and prove the following mirror theorem for coisotropic A-branes in the current setting of $A_n$-resolutions.

\begin{theorem}\label{thm:main theorem}
    For every $k \in \mathbb{Z}_{>0}$, the instanton-corrected mirror transform is a graded algebra isomorphism
    \begin{equation*}
        \Phi^{(k)}: \Hom_A(\cu{B}_{cc}^{(k)},\cu{B}_{cc}^{(k)}) \overset{\cong}{\longrightarrow}  \Hom_B(\check{\cu{B}}_{cc}^{(k)},\check{\cu{B}}_{cc}^{(k)}).
    \end{equation*}
\end{theorem}

As in the semi-flat case, the construction of $\Phi^{(k)}$ is local in the base direction. However, rather than relying on the original local construction in \eqref{eqn: semi-flat mirror transform} from \cite{ChaLeuLiSueYau2026}, we obtain the local map by identifying generators on both sides. The global construction then follows from the compatibility of this local description with the wall-crossing gluing data on both sides, which also explains the terminology ``instanton-corrected'' mirror transform.

To introduce the local construction of $\Phi^{(k)}$, we begin by characterizing holomorphic sections of $\operatorname{End}(\check{E}^{(k)})$ over $\check{X}_{U_i^\pm}^{(k)}$. Such sections may be identified with matrix-valued holomorphic functions on the universal cover $TU_i^\pm$ satisfying the natural equivariance condition with respect to the automorphy factor in \eqref{eqn: holomorphic automorphy factor}.

\begin{lemma}\label{lem:generation}
  Let $U\subset B^{sm}$ be an open subset so that the automorphy factor of $\check{E}^{(k)}|_{\check{X}^{(k)}_U}$ is given by
  $$\check{\tau}_U(\check{z}_1,\check{z}_2,m)=\exp\left(\frac{\pi}{2}(m_1\check{z}_2-m_2\check{z}_1)\right)A_m.$$
  Then any element $\Psi \in H^0(\check{X}^{(k)}_U,\End(\check{E}^{(k)}))$ admits a unique expansion of the form
  $$\Psi(\check{z}) = \sum_{a,b=0}^{2k-1}f_{ab}(\check{u},\check{w})e^{\pi(a\check{z}_2-b\check{z}_1)}S^aC^b,$$
  where $f_{ab}\in H^0(\check{X}^{(k)}_U,\cu{O}_{\check{X}^{(k)}})$ are holomorphic functions.
\end{lemma}
\begin{proof}
    Since $\{ S^aC^b \}_{a, b = 0}^{2k-1}$ forms a basis of the vector space of $2k\times 2k$ matrices over $\mathbb{C}$, $\Psi$ can be written uniquely as
    $$\Psi(\check{z})=\sum_{a,b=0}^{2k-1}\Psi_{ab}(\check{z})S^aC^b,$$
    for holomorphic functions $\Psi_{ab}$ on $TU$. Since $\Psi\in H^0(\check{X}^{(k)}_U,\End(\check{E}^{(k)}))$, it satisfies the equivariance condition
    $$\Psi\left(\check{z}+\frac{\sqrt{-1}}{k}m\right)A_m =A_m \Psi(\check{z}).$$
    Using $SC=qCS$ and comparing the coefficient of $C^aS^b$, we have
    $$\Psi_{ab}\left(\check{z}+\frac{\sqrt{-1}}{k}m\right)=q^{bm_1-am_2}\Psi_{ab}(\check{z}).$$
    Hence
    $$\Psi_{ab}(\check{z})=f_{ab}(\check{z})e^{\pi(a\check{z}_2-b\check{z}_1)},$$
    with $f_{ab}$ being $\frac{\sqrt{-1}}{k}\bb{Z}^2$-periodic, i.e. $f_{ab}(\check{z})=f_{ab}(\check{u},\check{w})$. This completes the proof of the lemma.
\end{proof}

For each $U = U_i^\pm$, we define maps
\begin{equation}
    \label{eqn: local mirror map}
    \Phi^{(k)}_U:H^0(X_U,\cu{O}_{X_U})\to H^0(\check{X}^{(k)}_U,\End(\check{E}^{(k)}))
\end{equation}
    by declaring that
    \begin{align*}
        \Phi^{(k)}_U(u):=&\,e^{\pi\check{z}_1}C^{-1},\\
        \Phi^{(k)}_U(w):=&\,e^{\pi\check{z}_2}S,
    \end{align*}
    and extending multiplicatively (with respect to the product $\star_{k^{-1}}$) and $\bb{C}$-linearly. The pairs $(u, w)$ and $(C^{-1}, S)$ satisfy analogous commutation relations: $u \star_{k^{-1}} w = qw \star_{k^{-1}} u$ and $C^{-1}S=qSC^{-1}$. It follows that
    $$\Phi^{(k)}_U:(H^0(X_U,\cu{O}_{X_U}),\star_{k^{-1}})\to(H^0(\check{X}^{(k)}_U,\End(\check{E}^{(k)})),\circ)$$
    is a well defined algebra homomorphism.
    
    Note that functions $f(\check{u},\check{w})$ are generated by $(e^{\pi\check{z}_1}C)^{2k},(e^{\pi\check{z}_2}S^{-1})^{2k}$ since $C^{2k}=S^{2k}=\Id$. By Lemma \ref{lem:generation}, we deduce that $\Phi^{(k)}_U$ are bijective, giving us a collection of algebra isomorphisms.
    
    \begin{remark}
        In Appendix \ref{app:Toeplitz}, we show that the above construction recovers the local mirror transform of \cite{ChaLeuLiSueYau2026}. The present formulation, however, provides a considerably more transparent description, which is particularly well suited to incorporating wall-crossing corrections.
    \end{remark}
    
    The global construction of $\Phi^{(k)}$ will be carried out as part of the proof of our main theorem.

\begin{proof}[Proof of Theorem \ref{thm:main theorem}]

    We verify that the local isomorphisms $\Phi_{U_i^{\pm}}^{(k)}$'s are compatible with the wall-crossing gluing data and hence glue to a globally defined construction. In other words, we need to show that
    \begin{equation}\label{eqn:nontrivial gluing equation}
        \check{G}_{U_i^-U_i^+}(\check{z}_2)\cdot\Phi^{(k)}_{U_i^+}(u_{U_i^+})\left(\check{z}_{U_i^-,1}+\frac{1}{2\pi k}\log(1+\check{w}^{-1}),\check{z}_2\right)\cdot\check{G}_{U_i^-U_i^+}(\check{z}_2)^{-1}=\Phi^{(k)}_{U_i^-}(u_{U_i^+})(\check{z}_{U_i^-,1},\check{z}_2),
    \end{equation}
    and
    \begin{equation}\label{eqn:gluing equation}
        \check{G}_{U_{i+1}^-U_i^+}(\check{z}_2)\cdot\Phi^{(k)}_{U_i^+}(u_{U_i^+})\left(\check{z}_{U_{i+1}^-,1},\check{z}_2\right)\cdot\check{G}_{U_{i+1}^-U_i^+}(\check{z}_2)^{-1}=\Phi^{(k)}_{U_{i+1}^-}(u_{U_i^+})(\check{z}_{U_{i+1}^-,1},\check{z}_2),
    \end{equation}
    where on the right hand side,
    $$u_{U_i^+}=\begin{dcases}
        u_{U_i^-}(1+w^{-2k})^{\frac{1}{2k}} & \text{ on }X_{B_i^+},\\
        u_{U_i^-}w^{-1}(1+w^{2k})^{\frac{1}{2k}}=q^{\frac{1}{2}}u\star_{k^{-1}} w^{-1}(1+w^{2k})^{\frac{1}{2k}} & \text{ on }X_{B_i^-},\\
        u_{U_{i+1}^-} & \text{ for all }i.
    \end{dcases}$$
    Equation \eqref{eqn:gluing equation} trivially holds on $\check{X}^{(k)}_{U_i^+\cap U_{i+1}^-}$ for all $i$, so we only need to prove \eqref{eqn:nontrivial gluing equation} on $\check{X}^{(k)}_{U_i^+\cap U_i^-}=\check{X}^{(k)}_{B_i^+}\sqcup\check{X}^{(k)}_{B_i^-}$.
    
    On $\check{X}^{(k)}_{B_i^+}$, Equation \eqref{eqn:nontrivial gluing equation} simplifies to
    $$\Phi^{(k)}_{U_i^+}(u_{U_i^+})\left(\check{z}_{U_i^-,1}+\frac{1}{2\pi k}\log(1+\check{w}^{-1}),\check{z}_2\right)=\Phi^{(k)}_{U_i^-}(u_{U_i^+})(\check{z}_{U_i^-,1},\check{z}_2),$$
    since $\check{G}_{U_i^-U_i^+}(\check{z}_2)$ is a scalar. We compute that
    $$\Phi^{(k)}_{U_i^+}(u_{U_i^+})\left(\check{z}_{U_i^-,1}+\frac{1}{2\pi k}\log(1+\check{w}^{-1}),\check{z}_2\right)=e^{\pi\check{z}_1}(1+\check{w}^{-1})^{\frac{1}{2k}}C^{-1},$$
    and
    $$\Phi^{(k)}_{U_i^-}(u_{U_i^+})(\check{z}_{U_i^-,1},\check{z}_2)=\Phi^{(k)}_{U_i^-}(u_{U_i^-})(\Id+\Phi^{(k)}_{U_i^-}(w)^{-2k})^{\frac{1}{2k}}=e^{\pi\check{z}_1}(1+\check{w}^{-1})^{\frac{1}{2k}}C^{-1}$$
    as $S^{2k}=\Id$ and $\check{w}=e^{2\pi k\check{z}_2}$. This gives Equation \eqref{eqn:nontrivial gluing equation} on $\check{X}^{(k)}_{B_i^+}$.

    On $\check{X}^{(k)}_{B_i^-}$, we need to show that
    $$T\cdot\Phi^{(k)}_{U_i^+}(u_{U_i^+})\left(\check{z}_{U_i^-,1}+\frac{1}{2\pi k}\log(1+\check{w}^{-1}),\check{z}_2\right)\cdot T^{-1}=\Phi^{(k)}_{U_i^-}(u_{U_i^+})(\check{z}_{U_i^-,1},\check{z}_2),$$
    where $T$ is defined in \eqref{eqn: definition of T}. 
    Using the fact that
    $$TC^{-1}=q^{\frac{1}{2}}C^{-1}S^{-1}T,$$
    we have
    $$T\cdot\Phi^{(k)}_{U_i^+}(u_{U_i^+})\left(\check{z}_{U_i^-,1}+\frac{1}{2\pi k}\log(1+\check{w}^{-1}),\check{z}_2\right)\cdot T^{-1}=q^{\frac{1}{2}}e^{\pi(\check{z}_1-\check{z}_2)}(1+\check{w})^{\frac{1}{2k}}C^{-1}S^{-1}.$$
    On the other hand,
    $$\Phi^{(k)}_{U_i^-}(u_{U_i^+})=q^{\frac{1}{2}}\Phi^{(k)}_{U_i^-}(u_{U_i^-})\cdot\Phi^{(k)}_{U_i^-}(w)^{-1}\cdot(\Id+\Phi^{(k)}_{U_i^-}(w)^{2k})^{\frac{1}{2k}}=q^{\frac{1}{2}}e^{\pi\check{z}_1}C^{-1}e^{-\pi\check{z}_2}S^{-1}(1+\check{w})^{\frac{1}{2k}},$$
    which is the same as $T\cdot\Phi^{(k)}_{U_i^+}(u_{U_i^+})\left(\check{z}_{U_i^-,1}+\frac{1}{2\pi k}\log(1+\check{w}^{-1}),\check{z}_2\right)\cdot T^{-1}$.

    Finally, using the refinement $\{U_{i,j}^{\pm}\}$ as shown in Figure \ref{fig:refined_charts} again, $\check{X}^{(k)}_{U_{i,j}^{\pm}}$ are also Stein and therefore, by Lemma \ref{lem:A isomorphism} and \ref{lem:B isomorphism}, $\Phi^{(k)}$ descends to a graded algebra isomorphism
    $$\Phi^{(k)}:(H^{\bullet}(B,\cu{A}_X^{(k)}),\star_{k^{-1}})\to(H^{\bullet}(\check{X}^{(k)},\End(\check{E}^{(k)})),\circ),$$
    proving our main theorem.
\end{proof}


\appendix
\section{The mirror transform is locally twisted family Toeplitz operators}\label{app:Toeplitz}
In this appendix, we prove that our mirror transform, locally, agrees with the construction given in \cite{ChaLeuLiSueYau2026}.

Let $(\check{E}_0^{(k)}, \nabla^{\check{E}_0^{(k)}})$ be the mirror bundle of $\cu{B}_{cc, 0}^{(k)}$ constructed in \cite{ChaLeuLiSueYau2026}. To simplify notation, we will often suppress the superscript $(k)$. Consider the open subset $U = U_i^\pm \subset B$. We briefly recall the local description of $(\check{E}_0 |_{\check{X}_U}, \nabla^{\check{E}_0} |_{\check{X}_U})$ from \cite{ChaLeuLiSueYau2026}. 
In this local model, smooth sections of $\check{E}_0 |_{\check{X}_U}$ can be identified with smooth $\mathbb{C}$-valued functions on $U \times \mathbb{R}$ that are rapidly decreasing in the $\mathbb{R}$-direction.

More precisely, we denote by $\mathcal{C}^\infty(U, \mathcal{S}(\mathbb{R}))$ the space of functions $s \in \mathcal{C}^\infty(U \times \mathbb{R}, \mathbb{C})$ such that, for all compact subset $K \subset U$ and all multi-indices $\mu \in \mathbb{N}^2$ and integers $\nu, \eta \in \mathbb{N}$,
	\begin{equation*}
		\lVert s \rVert_{K, \mu, \nu, \eta} := \sup_{(x, \check{y}_1) \in K \times \mathbb{R}} \left\lvert y^\eta \cdot \frac{\partial^{\lvert \mu \rvert + \lvert \nu \rvert} s}{\partial x^\mu \partial \check{y}_1^\nu} (x, \check{y}_1) \right\rvert < \infty.
	\end{equation*}
    Here, $\check{y}_1$ denotes the standard coordinate on the $\mathbb{R}$-factor. We equip $\mathcal{C}^\infty(U, \mathcal{S}(\mathbb{R}))$ with the locally convex topology generated by the seminorms $\lVert \,\cdot\, \rVert_{K, \mu, \nu, \eta}$. It is naturally a $\mathcal{C}^\infty(U, \mathbb{C})$-module via pointwise multiplication in the $U$-variable.
    
    By \cite[Proposition 4.6]{ChaLeuLiSueYau2026}, there exists a $\mathcal{C}^\infty(U, \mathbb{C})$-linear isomorphism of topological vector spaces
\begin{equation}
			\label{Equation: charaterization of sections of mirror brane}
			\Gamma(\check{X}_U, \check{E}_0) \to \mathcal{C}^\infty(U, \mathcal{S}(\mathbb{R})), \quad s \mapsto \langle s \rangle,
\end{equation}
where $\Gamma(\check{X}_U, \check{E}_0)$ is equipped with the standard compact-open $\mathcal{C}^\infty$-topology. 
Via this identification, the natural $\mathcal{C}^\infty(\check{X}_U, \mathbb{C})$-module structure on $\Gamma(\check{X}_U,\check{E}_0)$ is transported to a $\mathcal{C}^\infty(\check{X}_U, \mathbb{C})$-module structure on $\mathcal{C}^\infty(U,\mathcal{S}(\mathbb{R}))$, extending its natural $\mathcal{C}^\infty(U, \mathbb{C})$-module structure. This action is determined by the Fourier modes
\begin{equation}
\label{eqn: fibrewise Fourier mode}
    \check{F}_m(x,\check{y})
:=
e^{2\pi\sqrt{-1}(m_1\check{y}_1+m_2\check{y}_2)},
\quad
m=(m_1,m_2)\in\mathbb Z^2.
\end{equation}
Indeed, for every $s\in\Gamma(\check{X}_U,\check{E}_0)$ and $m\in\mathbb Z^2$, one has
\begin{equation}
    \label{eqn: multiplication by Fourier modes}
    \langle \check{F}_m \cdot s \rangle (x, \check{y}_1) = e^{2\pi\sqrt{-1} m_1\check{y}_1} \cdot \langle s \rangle(x, \check{y}_1 + 2km_2).
\end{equation}

We now show that the above description is equivalent to the approach adopted in this article. Consider the trivial Hermitian vector bundle $TU \times \mathbb{C}^{2k}$ of rank $2k$ equipped with the unitary connection $\check{\nabla} = d + \alpha$, where
\begin{equation*}
    \alpha := -\pi\sqrt{-1}\left( k x_1dx_2 + \frac{1}{k} \check{y}_1 d\check{y}_2 \right) \Id_{2k} \in \Omega^1(TU, \mathfrak{u}(2k)).
\end{equation*}
Let $\check{\tau}^{\text{uni}}(x, \check{y}, m)$ denote the unitary automorphy factor defined in \eqref{eqn: unitary automorphy factor}. It is compatible with the connection $1$-form $\alpha$ in the sense that for all $m \in \mathbb{Z}^2$,
\begin{align*}
    & \alpha(x, \check{y} + m) - \check{\tau}^{\text{uni}} (x, \check{y}, m) \alpha(x, \check{y}) (\check{\tau}^{\text{uni}}(x, \check{y}, m))^{-1} = -(d\check{\tau}^{\text{uni}}(x, \check{y}, m)) (\check{\tau}^{\text{uni}}(x, \check{y}, m))^{-1}.
\end{align*}
Consequently, the quotient Hermitian vector bundle $\mathbf{\check{E}}$ is well defined, and $\check{\nabla}$ descends to a unitary connection on $\mathbf{\check{E}}$, which we denote by the same symbol (cf. \cite[Proposition 6.1.3]{Sis2024}). We identify $\Gamma(\check{X}_U, \mathbf{\check{E}})$ with the space of functions $s \in \mathcal{C}^\infty(TU, \bb{C}^{2k})$ satisfying the equivariance condition:
\begin{equation}
    \label{eqn: equivariance condition}
    s(x, \check{y} + m) = \check{\tau}^{\text{uni}} (x, \check{y}, m) s(x, \check{y}),
\end{equation}
for all $m \in \mathbb{Z}^2$. Here, $\check{\tau}^{\text{uni}} (x, \check{y}, m)$ acts on $s(x, \check{y})$ by matrix multiplication.

\begin{proposition}
    For each $k \in \mathbb{Z}_{>0}$, there is a $\mathcal{C}^\infty(\check{X}_U^{(k)})$-linear isomorphism
    \begin{equation}
    \label{eqn: identification of local models}
    \Gamma(\check{X}_U^{(k)}, \check{E}_0^{(k)}) \to \Gamma(\check{X}_U^{(k)}, \mathbf{\check{E}}^{(k)}), \quad s \mapsto [s],
\end{equation}
such that for all $V \in \Gamma(\check{X}_U^{(k)}, T\check{X}_U^{(k)})$ and $s \in \Gamma(\check{X}_U^{(k)}, \check{E}_0^{(k)})$,
\begin{equation*}
    \left[ \nabla_V^{\check{E}_0^{(k)}} s \right] = \check{\nabla}_V^{(k)} [s].
\end{equation*}
\end{proposition}
\begin{proof}
For each $s \in \Gamma(\check{X}_U, \check{E}_0)$, define $[s] = ([s]_0, ..., [s]_{2k-1}) \in \mathcal{C}^\infty(TU, \mathbb{C}^{2k})$, where for all $l = 0, \dots, 2k-1$,
\begin{equation*}
        [s]_l(x, \check{y}) := \sum_{p \in \mathbb{Z}} \langle s \rangle ( x, \check{y}_1 - l - 2kp) \cdot e^{\frac{\pi\sqrt{-1}}{k} (l + 2kp) \check{y}_2}.
\end{equation*}
This construction can be viewed as a family version of the correspondence between sections of a $U(2k)$-bundle with topological charge $1$ over a $2$-torus and Schwartz functions on $\bb{R}$, as described in \cite[page 9]{AscDes2019}, 
up to a minor modification. For all $l = 0, \dots, 2k-1$ and $m \in \mathbb{Z}^2$, a direct computation yields
\begin{align*}
    [s]_l(x, \check{y} + m)
    = & q^{-m_1m_2}  e^{\frac{\pi\sqrt{-1}}{k} m_1 \check{y}_2} q^{-m_2 l'} [s]_{l'}(x, \check{y}),
\end{align*}
where $l' \in \{0, ..., 2k-1\}$ is the unique index satisfying $l' \equiv l-m_1 \pmod {2k}$. 
Therefore, $[s]$ satisfies the equivariance condition \eqref{eqn: equivariance condition}, 
and hence $[s] \in \Gamma(\check{X}_U, \mathbf{\check{E}})$. Following \cite{AscDes2019} together with \cite[Proposition 4.6]{ChaLeuLiSueYau2026}, it is straightforward to verify that the assignment $s \mapsto [s]$ defines a $\mathcal{C}^\infty(U, \mathbb{C})$-linear isomorphism
$$\Gamma(\check{X}_U, \check{E}_0) \cong \Gamma(\check{X}_U, \mathbf{\check{E}}).$$
We next verify that this isomorphism is $\mathcal{C}^\infty(\check{X}_U, \mathbb{C})$-linear. It suffices to check compatibility with multiplication by the Fourier modes $\check{F}_m$ stated in \eqref{eqn: fibrewise Fourier mode}. For all $l = 0, \dots, 2k-1$ and $m \in \mathbb{Z}^2$, we compute
\begin{align*}
    [\check{F}_m \cdot s]_l(x, \check{y})
    = & \sum_{p \in \mathbb{Z}} e^{2\pi\sqrt{-1} m_1\check{y}_1} \langle s \rangle ( x, \check{y}_1 - l - 2kp+2km_2) \cdot e^{\frac{\pi\sqrt{-1}}{k} (l + 2kp) \check{y}_2},
\end{align*}
using \eqref{eqn: multiplication by Fourier modes}. Shifting the summation index $p$ by $m_2$, we obtain $[\check{F}_m \cdot s]_l = \check{F}_m \cdot [s]_l$.

Finally, using the computations in the proof of \cite[Theorem 4.1]{ChaLeuLiSueYau2026}, we obtain
		\begin{align*}
            \left\langle \nabla_{\partial_{x_1}}^{\check{E}_0} s \right\rangle = \left( \frac{\partial}{\partial x_1} \right) \langle s \rangle, \quad & \left\langle \nabla_{\partial_{x_2}}^{\check{E}_0} s \right\rangle = \left( \frac{\partial}{\partial x_2} - \pi\sqrt{-1} k x_1 \right) \langle s \rangle,\\
            \left\langle \nabla_{\partial_{\check{y}_1}}^{\check{E}_0} s \right\rangle = \left( \frac{\partial}{\partial \check{y}_1} \right) \langle s \rangle, \quad &  \left\langle \nabla_{\partial_{\check{y}_2}}^{\check{E}_0} s \right\rangle = \left( -\pi\sqrt{-1} \cdot \frac{1}{k} \check{y}_1 \right) \langle s \rangle.
		\end{align*}
        Consequently, $\left[ \nabla_V^{\check{E}_0} s \right] = \check{\nabla}_V [s]$ holds for each coordinate vector field $V = \partial_{x_1}, \partial_{x_2}, \partial_{\check{y}_1}, \partial_{\check{y}_2}$. By $\mathcal{C}^\infty(\check{X}_U)$-linearity, the identity extends to all vector fields $V \in \Gamma(\check{X}_U, T\check{X}_U)$.
\end{proof}

In \cite[Proposition 5.6]{ChaLeuLiSueYau2026}, the authors, together with Chan--Leung--Li, constructed a $\mathcal{C}^\infty(U, \mathbb{C})$-linear map
$$\widetilde{\Phi}^{(k)}: \mathcal{C}^\infty(X_U, \mathbb{C}) \to \Gamma(\check{X}_U^{(k)}, \operatorname{End}(\check{E}_0^{(k)})),$$
which provides the local model for the semi-flat mirror transform as stated in \eqref{eqn: semi-flat mirror transform}. By \cite[Lemmas 5.10--5.12]{ChaLeuLiSueYau2026}, the map $\widetilde{\Phi}^{(k)}$ restricts to an isomorphism
$$\mathcal{O}(X_U) \cong H^0(\check{X}_U^{(k)}, \operatorname{End}(\check{E}_0^{(k)})).$$
We conclude this appendix by showing that, under the identification \eqref{eqn: identification of local models}, this restriction agrees with the local mirror transform $\Phi_U^{(k)}$ defined in \eqref{eqn: local mirror map}. Since both maps are algebra homomorphisms, it suffices to verify that they coincide on the monomials $u$ and $w$. The general case then follows by continuity with respect to the relevant topologies.


        \begin{proposition}
        For all $k \in \mathbb{Z}_{>0}$ and $s \in \Gamma(\check{X}_U^{(k)}, \check{E}_0^{(k)})$,
        \begin{align*}
            [\widetilde{\Phi}^{(k)}(u)(s)](x, \check{y}) = & \, e^{\pi \check{z}_1} C^{-1}[s](x, \check{y}),\\
            [\widetilde{\Phi}^{(k)}(w)(s)](x, \check{y}) = & \, e^{\pi \check{z}_2} S[s](x, \check{y}).
        \end{align*}
        \end{proposition}
        \begin{proof}
        Applying \cite[Proposition 5.6]{ChaLeuLiSueYau2026} to Fourier modes $e^{2\pi\sqrt{-1} y^1}$ and $e^{-2\pi\sqrt{-1} y^2}$, we obtain
        \begin{align*}
            \langle \widetilde{\Phi}^{(k)}(e^{2\pi\sqrt{-1} y^1})(s) \rangle(x, \check{y}_1) = &\, \langle s \rangle(x, \check{y}_1+1),\\
			\langle \widetilde{\Phi}^{(k)}(e^{-2\pi\sqrt{-1} y^2})(s) \rangle(x, \check{y}_1) = &\, e^{\frac{2\pi \sqrt{-1}}{k} \check{y}_1} \cdot \langle s \rangle(x, \check{y}_1).
		\end{align*}
        Recall that $u = e^{\pi(x_1 - 2\sqrt{-1}y^2)}$ and $w = e^{\pi(x_2 + 2\sqrt{-1}y^1)}$. By $\mathcal{C}^\infty(U, \mathbb{C})$-linearity, for all $l = 0, \dots, 2k-1$,
        \begin{align*}
            [\widetilde{\Phi}^{(k)}(u)(s)]_l(x, \check{y}) = \sum_{p \in \mathbb{Z}} e^{\pi (\check{z}_1 - \frac{\sqrt{-1}}{k}(l+2kp))} \cdot \langle s \rangle (x, \check{y}_1 - l - 2kp) \cdot e^{\frac{\pi\sqrt{-1}}{k}(l+2kp) \check{y}_2} = e^{\pi \check{z}_1} q^{l'} [s]_{l'} (x, \check{y}).
        \end{align*}
        In the second equality, we used $\exp \left(\pi \cdot \left(-\tfrac{\sqrt{-1}}{k}\right) \cdot 2kp\right) = 1$.

        By $\mathcal{C}^\infty(U, \mathbb{C})$-linearity again, for each $l = 0, \dots, 2k-1$,
        \begin{align*}
            [\widetilde{\Phi}^{(k)}(w)(s)]_l(x, \check{y}) = & \sum_{p \in \mathbb{Z}} e^{\pi x_2} \cdot \langle s \rangle (x, \check{y}_1 + 1 - l - 2kp) \cdot e^{\frac{\pi\sqrt{-1}}{k}(l+2kp) \check{y}_2}\\
            = & e^{\pi \check{z}_2} \sum_{p \in \mathbb{Z}} \langle s \rangle (x, \check{y}_1 + 1 - l - 2kp) \cdot e^{\frac{\pi\sqrt{-1}}{k}(-1+l+2kp) \check{y}_2}.
        \end{align*}
        Reindexing the summation appropriately, we obtain, we obtain
        $$[\widetilde{\Phi}^{(k)}(w)(s)]_l(x, \check{y}) = e^{\pi \check{z}_2} [s]_{l'} (x, \check{y}),$$
        where $l' \in \{0, ..., 2k-1\}$ is the unique index satisfying $l' \equiv l-1 \pmod {2k}$. This completes the proof.
        \end{proof}

        In Section \ref{sec:mirror theorem}, the expressions for $\Phi_U^{(k)}(u)$ and $\Phi_U^{(k)}(w)$ are given as matrix-valued functions with respect to the holomorphic frame $e^{(\alpha)}$, whereas the matrix expressions for $\widetilde{\Phi}^{(k)}(u)$ and $\widetilde{\Phi}^{(k)}(w)$ in the preceding proposition are written with respect to the frame $1^{(\alpha)}$. However, these two choices of frames yield the same matrix expressions in this case. We therefore conclude, by the above proposition, that the restriction of $\widetilde{\Phi}^{(k)}$ coincides with $\Phi_U^{(k)}$.
        
	\bibliographystyle{amsplain}
	\bibliography{References}
	
\end{document}